\documentclass[12pt]{amsart}
\usepackage{amscd,amssymb,amsthm,verbatim}
\newcommand{\ad}{\operatorname{ad}}
\newcommand{\Ad}{\operatorname{Ad}}

\newcommand{\APS}{\operatorname{APS}}

\newcommand{\Aut}{\operatorname{Aut}}
\newcommand{\bas}{\operatorname{bas}}

\newcommand{\C}{{\mathbb C}}

\newcommand{\ch}{\operatorname{ch}}

\newcommand{\Dirac}{\operatorname{Dirac}}

\newcommand{\Dom}{\operatorname{Dom}}

\newcommand{\End}{\operatorname{End}}

\newcommand{\HH}{\operatorname{H}}

\newcommand{\Id}{\operatorname{Id}}

\newcommand{\Image}{\operatorname{Im}}

\newcommand{\Ind}{\operatorname{Ind}}
\newcommand{\Index}{\operatorname{Index}}

\newcommand{\inv}{\operatorname{inv}}

\newcommand{\Isom}{\operatorname{Isom}}

\newcommand{\Ker}{\operatorname{Ker}}

\newcommand{\R}{{\mathbb R}}

\newcommand{\sign}{\operatorname{sign}}

\newcommand{\Spin}{\operatorname{Spin}}

\newcommand{\str}{\operatorname{tr}_s}
\newcommand{\Str}{\operatorname{Tr}_s}
\newcommand{\supp}{\operatorname{supp}}

\newcommand{\tr}{\operatorname{tr}}
\newcommand{\trs}{\operatorname{tr}_s}
\newcommand{\Tr}{\operatorname{Tr}}

\newcommand{\vol}{\operatorname{vol}}
\newcommand{\Weyl}{\operatorname{Weyl}}
\newcommand{\Z}{{\mathbb Z}}

\numberwithin{equation}{section}
\theoremstyle{plain}

\newtheorem{assumption}{Assumption}
\newtheorem{lemma}{Lemma}
\newtheorem{theorem}{Theorem}
\newtheorem{proposition}{Proposition}
\newtheorem{corollary}{Corollary}

\theoremstyle{definition}

\newtheorem{remark}{Remark}
\newtheorem{example}{Example}
\newtheorem{definition}{Definition}
\begin{document}
\title{The index of a transverse Dirac-type operator : the case of abelian Molino sheaf}
\author{Alexander Gorokhovsky and John Lott}
\address{Department of Mathematics\\
University of Colorado\\
Boulder, CO  80309-0395\\
USA}
\email{Alexander.Gorokhovsky@colorado.edu}
\address{Department of Mathematics\\
University of California - Berkeley\\
Berkeley, CA  94720-3840\\
USA} \email{lott@math.berkeley.edu}

\thanks{The research of the first author was partially supported by NSF grant
DMS-0900968. The research of the second author was partially supported
by NSF grant DMS-0903076}
\date{August 11, 2011}

\begin{abstract}
We give a local formula for the index of a transverse Dirac-type operator on
a compact manifold with a Riemannian foliation, under the assumption that the
Molino sheaf is a sheaf of abelian Lie algebras.
\end{abstract}

\maketitle


\section{Introduction} \label{section1}

An important test case for noncommutative geometry comes from index theory on
compact foliated manifolds, as pioneered by Connes and his collaborators.  The most commonly
considered case is that of a leafwise
Dirac-type operator $D$.  Its index $\Index(D)$ lies in the $K$-theory of a
stabilized version of the
foliation groupoid algebra.  The local index theorem gives an explicit formula for the pairing
of $\Index(D)$ with cyclic cohomology classes of the foliation groupoid algebra. For
more information on this well-developed theory, we refer to
\cite{Connes (1994),Connes-Skandalis (1984),Gorokhovsky-Lott (2003),Gorokhovsky-Lott (2006)}.

This paper is concerned with a different index problem for compact foliated manifolds, namely
that of a transverse Dirac-type operator. Such an operator differentiates in
directions normal to the leaves. In order to make sense of the operator, we
must assume that the foliation is Riemannian, i.e. the normal bundle to the leaves
carries a holonomy-invariant inner product.
Then there is a notion of a
``basic'' Dirac-type operator $D$, a first-order differential
operator that acts on the
holonomy-invariant sections of a normal Clifford module.
It was shown by El Kacimi \cite{El-Kacimi (1990)}
and Glazebrook-Kamber \cite{Glazebrook-Kamber (1991)}
that $D$ is Fredholm
and hence has a well-defined index $\Index(D) \in \Z$.
(In fact, this is true for any basic transversally elliptic operator \cite{El-Kacimi (1990)}.)
 The index problem, which has been open
for twenty years \cite[Probl\`eme 2.8.9]{El-Kacimi (1990)}, is to give an explicit
formula for $\Index(D)$, in terms of the Riemannian foliation.
A prototypical example is that of a compact manifold foliated by points, in
which case the index is given by the Atiyah-Singer formula.

From the noncommutative geometry viewpoint, a leafwise Dirac-type operator is a family of
Dirac-type operators parametrized by the ``leaf space'' of the foliation, where the
``leaf space'' is defined in terms of algebras.  In contrast, a transverse Dirac-type operator
is a differential operator on such a ``leaf space''.  As will be seen,
the transverse index problem can be usefully formulated in terms of Riemannian groupoids.
Such groupoids also arose in work of Petrunin-Tuschmann on the collapsing theory of Riemannian manifolds
\cite{Petrunin-Tuschmann (1999)} and work of the second author on Ricci flow \cite{Lott (2007),Lott (2010)}.
Our interest in the transverse index problem comes from the
more general program of doing analysis on Riemannian groupoids.

To a Riemannian foliation ${\mathcal F}$ on a compact connected manifold $M$,
one can canonically associate a locally constant sheaf
of Lie algebras on $M$, called the Molino sheaf \cite{Molino (1988)}. Let
${\mathfrak g}$ denote the finite-dimensional Lie algebra which appears as the stalk of
the Molino sheaf. ({\it A priori}, it could be any finite-dimensional Lie algebra.)
If ${\mathfrak g} = 0$, which happens if and only if the leaves are compact,
then the leaf space is an orbifold and the
transverse index theorem reduces to Kawasaki's orbifold index theorem
\cite{Kawasaki (1978),Kawasaki (1981)}. In this paper we give the first local formula for
$\Index(D)$ in a case when ${\mathfrak g} \neq 0$. The case that we consider is when ${\mathfrak g}$ is an abelian Lie
algebra
$\R^k$.

To state our index theorem, we recall some information about Riemannian foliations.
Although the leaf space of such a foliation may be pathological (for example
non-Hausdorff), the space $W$
of leaf closures is a nice Hausdorff space which is stratified by manifolds.
A neighborhood of a point $w \in W$ is homeomorphic to $V_w/K_w$, where
$K_w$ is a compact Lie group that is canonically associated to $w$, and
$V_w$ is a representation space of $K_w$.

\begin{assumption} \label{assumption1}
1. The Molino Lie algebra is an abelian Lie algebra
$\R^k$.
\\
2. The Molino sheaf has trivial holonomy on $M$. \\
3. For all $w \in W$, the group $K_w$ is connected.
\end{assumption}

Here Assumptions \ref{assumption1}.1 and \ref{assumption1}.2 automatically hold if $M$ is simply-connected.

If Assumption \ref{assumption1} holds then $K_w$ is isomorphic to $T^{j_w}$
for some $j_w \in [0,k]$.
Put $W_{\max} = \{ w \in W \: : \: K_w \cong T^k \}$. Then $W_{\max}$ is a
smooth manifold which is the deepest stratum of $W$.
Note that $W_{\max}$ may be the empty set.

\begin{theorem} \label{theorem1}
Let $M$ be a compact connected manifold equipped with a Riemannian foliation
${\mathcal F}$. Suppose that Assumption \ref{assumption1} holds.
Let ${\mathcal E}$ be a holonomy-invariant normal Clifford module on $M$,
on which the Molino sheaf acts.
Let $D$ be the basic Dirac-type operator acting
on holonomy-invariant sections of ${\mathcal E}$. Then
\begin{equation} \label{1.1}
\Index(D) \: = \: \int_{W_{\max}} \widehat{A}(TW_{\max}) \:
{\mathcal N}_{{\mathcal E}, Q}.
\end{equation}
\end{theorem}

Here ${\mathcal N}_{{\mathcal E}, Q}$
is a ``renormalized'' characteristic class which is
computed from the normal
data of $W_{\max}$ along with the restriction of ${\mathcal E}$ to
$W_{\max}$.  More precisely, it arises by multiplying
the Atiyah-Singer
normal characteristic class and an equivariant Chern class for
${\mathcal E} \big|_{W_{\max}}$, and performing an averaging process;
see Definition \ref{generaldef}.
Because of the computability of ${\mathcal N}_{{\mathcal E}, Q}$,
we can derive the following consequences.

\begin{corollary} \label{corollary1}
Under Assumption \ref{assumption1},\\
1. The basic Euler characteristic of $(M, {\mathcal F})$ equals
the Euler characteristic of $W_{\max}$. \\
2. If ${\mathcal F}$ is transversely oriented then the basic signature of
$(M, {\mathcal F})$ equals
the signature of $W_{\max}$. \\
3. Suppose that ${\mathcal F}$ has a transverse spin structure. Let $D$ be the
basic Dirac operator.  Then
$\Index(D) = \widehat{A} \left( W_{\max} \right)$
if $k = 0$, while
$\Index(D) = 0$ if $k > 0$.
\end{corollary}

The proof of Theorem \ref{theorem1} requires some new techniques. To motivate these,
we start with a special case.
An especially tractable example of a Riemannian foliation comes from a suspension
construction, as described in Examples \ref{example1}-\ref{example8} and Section \ref{section4}.
In this case, the transverse structure can be described by the following data
:
\begin{enumerate}
\item A discrete finitely presented group $\Gamma$,
\item A compact Lie group $G$,
\item An injection $i \: : \: \Gamma \rightarrow G$ with dense image, and
\item A closed Riemannian manifold $Z$ on which $G$ acts isometrically.
\end{enumerate}

With this data, a transverse Dirac-type operator on the suspension foliation amounts to
a Dirac-type operator on $Z$ which is $\Gamma$-invariant or, equivalently,
$G$-invariant. In this case, the index problem amounts to computing the index of $D$, the
restriction of the Dirac-type operator to the $G$-invariant sections of the Clifford module.
Such an index can easily be computed as $\Index(D) = \int_G \Index(g) \: d\mu_G(g)$,
where $\Index(g) \in \R$ is the $G$-index and $d\mu_G$ is Haar measure on $G$.

The Atiyah-Singer $G$-index theorem \cite{Atiyah-Singer (1968)} tells us that
$\Index(g) \: = \: \int_{Z^g} {\mathcal L}(g)$, where $Z^g$ is the fixed-point set of $g \in G$ and
${\mathcal L}(g) \in \Omega^*(Z^g)$ is an explicit characteristic class.  Suppose that $G$ is a
torus group $T^k$.
After performing
the integral over $g \in T^k$, only the submanifolds with $Z^g = Z^{T^k}$ will contribute,
where $Z^{T^k}$ denotes the fixed-point set of $T^k$. Hence
we can write
\begin{equation} \label{1.2}
\Index(D) = \int_{T^k} \int_{Z^{T^k}} {\mathcal L}(g) \: d\mu_{T^k}(g).
\end{equation}

In order to give a local expression for $\Index(D)$, we would like to exchange integrals
and write
\begin{equation} \label{1.3}
\Index(D) \stackrel{?}{=} \int_{Z^{T^k}} \int_{T^k} {\mathcal L}(g) \: d\mu_{T^k}(g).
\end{equation}
But there is a surprise : the integral $\int_{T^k} {\mathcal L}(g) \: d\mu_{T^k}(g) \in
\Omega^*(Z^{T^k})$ generally diverges! The reason that (\ref{1.2}) makes sense is that there are
cancellations of singularities arising from the various connected components of $Z^{T^k}$.
After identifying these singularities (which will cancel in the end) one can subtract them by hand
and thereby
obtain a valid ``renormalized'' local index formula
\begin{equation} \label{1.4}
\Index(D) = \int_{Z^{T^k}} \widehat{A} \left( TZ^{T^k} \right) \:
{\mathcal N}.
\end{equation}

In general, the transverse structure of a Riemannian foliation does not admit a global
Lie group action like in the suspension case.  This is a problem for seeing the cancellation of singularities.
Instead, if the Molino sheaf has trivial holonomy then there is a
global {\em Lie algebra} action, by ${\mathfrak g}$. Because of this, we
use the Kirillov delocalized approach to equivariant index theory
\cite[Chapter 8]{Berline-Getzler-Vergne (1992)}. If ${\mathfrak g}$ is abelian then we can
replace the nonexistent ``integration over $G$'' by an averaging over ${\mathfrak g}$. In summary,
our proof of Theorem \ref{theorem1} combines a parametrix construction, using local models
for the transverse structure, with Kirillov-type equivariant index
formulas and an averaging over ${\mathfrak g}$.

It is not clear to us whether our methods extend beyond the restrictions in
Assumption \ref{assumption1}. If we remove Assumption \ref{assumption1}.3 then the analog of $W_{\max}$ is an
orbifold and the right-hand side of (\ref{1.1}) makes sense.  However, in this case it is not
clear  whether our proof extends if $k > 0$.

In this paper we focus on the transverse structure of the foliation, as
opposed to the leafwise structure.  More precisely, we choose a complete transversal
$Z$ for the foliation and work with the \'etale groupoid ${\mathcal G}_{\mathcal T}$
whose unit space is $Z$, as
opposed to the foliation groupoid whose unit space is $M$.  Let us mention an attractive alternative
approach to the transverse index theorem. It consists of passing to the normal frame bundle
$F_{O(q)} M$ of $M$, where one obtains an $O(q)$-transversally elliptic differential
operator. Atiyah showed that such an operator has an index which is a distribution on
$O(q)$ \cite{Atiyah (1974)}. The numerical index $\Index(D)$ is the result of pairing this distribution
with the identity function. There is an index formula for $G$-transversally elliptic operators,
due to Berline and Vergne
\cite{Berline-Vergne (1996),Paradan-Vergne (2008)}.
Unfortunately, this index formula is not explicit enough to yield
a local formula for $\Index(D)$. Consequently, we stick to the
Riemannian groupoid ${\mathcal G}_{\mathcal T}$, although we do use frame bundles for some
technical points.

Let us also mention that there is a transverse index theorem developed by
Br\"uning-Kamber-Richardson
\cite{Bruening-Kamber-Richardson1 (2010),Bruening-Kamber-Richardson2 (2010),Bruening-Kamber-Richardson3 (2010)}, based on doing analysis
on the singular space $W$. In this way they obtain an index formula
involving integrals over desingularizations of strata along with
eta-invariants of normal spheres.

The structure of this paper is as follows.  In Section \ref{section2}
we review material about Riemannian foliations and Riemannian
groupoids.  We discuss the groupoid closure and construct a
Haar system for it.  In Section \ref{section3} we describe basic
Dirac-type operators in the setting of spectral
triples.  We prove an isomorphism between the image of a
certain projection operator, acting on all smooth
sections of the transverse Clifford module, and
the space of
holonomy-invariant smooth sections of the transverse Clifford module.
We use this to define the invariant Dirac-type operator as a
self-adjoint operator.
In Section \ref{section4}, which can be read independently of the
rest of the paper, we consider the special case of a
Riemannian foliation which arises as the suspension of a
group of isometries of a compact manifold.
In Section \ref{section5} we specialize to the case of
abelian Molino sheaf. We construct a parametrix and prove
a delocalized index theorem.
In Section \ref{section6} we localize this result and
prove Theorem \ref{theorem1}. We also compute the
indices in some geometric examples.

More detailed descriptions can be found at the beginnings of
the sections.

We thank the referee for useful comments.

\section{Riemannian groupoids and their closures} \label{section2}

In this section we collect material, some of it well known and
some of it not so well known, about Riemannian foliations
and Riemannian groupoids.
For basic information about foliations and groupoids, we refer to
\cite{Moerdijk-Mrcun (2003)}. A survey on Riemannian foliations
is in \cite{Haefliger (1989)}.

In Subsection \ref{subsection2.1} we introduce some notation and basic
ideas about Riemannian groupoids.

It will be important for
us to be able to take the closure of a Riemannian groupoid,
in an appropriate sense.  This is because the closure is a
proper Lie groupoid, which allows for averaging.  Hence in
Subsection \ref{subsection2.2} we recall the construction of the groupoid
closure.  In order to do averaging, we need a Haar system on
the groupoid closure.  Our construction of the Haar system
is based on passing to the frame bundle of a transversal, which is
described in Subsection \ref{subsection2.3}.  Subsection \ref{subsection2.4} contains
the construction of the Haar system, along with certain
mean curvature one-forms.

In Subsection \ref{subsection2.5} we summarize
Molino theory in terms of the Lie algebroid of the groupoid closure.
Subsection \ref{subsection2.6} recalls Haefliger's local models for the
transverse structure of a Riemannian foliation.  Finally,
in Subsection \ref{subsection2.7}, we recall Sergiescu's dualizing sheaf
for a Riemannian groupoid and show how a square root of the
dualizing sheaf allows one to define a basic signature.

\subsection{Riemannian groupoids} \label{subsection2.1}
Suppose that ${\mathcal G}$ is a smooth effective \'etale groupoid
\cite[Chapter 5.5]{Moerdijk-Mrcun (2003)}. The space of units is denoted
${\mathcal G}^{(0)}$.
We will denote the source and range maps of ${\mathcal G}$ by $s$ and $r$, respectively.
Our conventions are that
$g_1 g_2$ is defined if and only if $s(g_1) = r(g_2)$.
We write ${\mathcal G}^p$ for $r^{-1}(p)$, ${\mathcal G}_p$ for $s^{-1}(p)$ and
${\mathcal G}^p_p$ for the isotropy group
$s^{-1}(p) \cap r^{-1}(p)$. For simplicity of notation, we 
write $g \in {\mathcal G}$ instead of $g \in {\mathcal G}^{(1)}$ when 
referring
to an element of the groupoid.
We write
$dg_{s(g)} : T_{s(g)}{\mathcal G}^{(0)} \rightarrow  T_{r(g)}{\mathcal G}^{(0)}$
for the linearization of $g$.

For us, an action of ${\mathcal G}$ on
a manifold $Z$ is a right action. That is, one first has
a submersion $\pi : Z \rightarrow {\mathcal G}^{(0)}$. Putting
\begin{equation} \label{2.1}
Z \times_{{\mathcal G}^{(0)}} {\mathcal G} = \{(p,g) \in Z \times {\mathcal G} \: : \:
\pi(p) = r(g)\},
\end{equation}
we must also have a smooth map $Z \times_{{\mathcal G}^{(0)}} {\mathcal G} \rightarrow Z$,
denoted $(p, g) \rightarrow pg$, such that $\pi(pg) = s(g)$ and $(pg_1)g_2 = p(g_1 g_2)$ for
all composable $g_1, g_2$. There is an associated cross-product groupoid
$Z \rtimes {\mathcal G}$ with $s(p,g) = pg$ and $r(p,g) = p$.

Our notion of
equivalence for smooth effective \'etale groupoids is {\em weak equivalence}
\cite[Chapter 5.4]{Moerdijk-Mrcun (2003)}, which is sometimes
called Morita equivalence.  (This is distinct from groupoid
isomorphism.) A useful way to characterize weak
equivalence (for \'etale groupoids) is the following
\cite[Exercise III.${\mathcal G}$.2.8(2)]{Bridson-Haefliger (1999)} :
two smooth \'etale groupoids ${\mathcal G}$ and ${\mathcal G}^\prime$
are weakly equivalent if there are open covers ${\mathcal U}$ and
${\mathcal U}^\prime$ of their unit spaces so that the
localizations ${\mathcal G}_{\mathcal U}$ and
${\mathcal G}^\prime_{{\mathcal U}^\prime}$ are isomorphic smooth groupoids.

A smooth \'etale groupoid ${\mathcal G}$ is {\em Riemannian}
if there is a Riemannian metric
on ${\mathcal G}^{(0)}$ so that the groupoid elements act by
local isometries. That is, for each $g \in {\mathcal G}$,
the map $dg_{s(g)}$ is an isometric isomorphism.
There is an evident notion of
isomorphism for Riemannian groupoids. Two Riemannian groupoids are
{\em equivalent} if there are localizations ${\mathcal G}_{\mathcal U}$ and
${\mathcal G}^\prime_{{\mathcal U}^\prime}$ which are isomorphic
Riemannian groupoids.

A Riemannian groupoid is {\em complete}
in the sense of \cite[Definition 3.1.1]{Haefliger (1985)} if for all $p_1, p_2 \in
{\mathcal G}^{(0)}$, there
are neighborhoods $U_1$ of $p_1$ and $U_2$ of $p_2$
so that any groupoid element $g$ with $s(g) \in U_1$ and $r(g) \in U_2$ has an extension
to all of $U_1$. That is, for any such $g$,
there is a smooth map $\tau : U_1 \rightarrow {\mathcal G}$ 
with $\tau(s(g)) = g$
and $s \circ \tau \: = \: \Id$.

\subsection{Groupoid closures} \label{subsection2.2}

Let $M$ be a connected closed $n$-dimensional manifold with a
codimension-$q$ foliation ${\mathcal F}$.
A {\em Riemannian foliation} structure on ${\mathcal F}$ is
an inner product on the normal bundle $TM/T{\mathcal F}$ which
is holonomy-invariant. See \cite[Chapter 2.2]{Moerdijk-Mrcun (2003)}
for some equivalent formulations.  In what follows, we assume that
${\mathcal F}$ has a fixed Riemannian foliation structure.

There is a partition of $M$ by the leaf closures.
The quotient space $W$ is Hausdorff but is generally not a manifold.

Let $F_{O(q)}M$ denote the orthonormal normal frame bundle to
${\mathcal F}$ \cite[Chapter 4.2.2]{Moerdijk-Mrcun (2003)}.
It has a lifted codimension-$q$ foliation
$\widehat{\mathcal F}$. The leaf closures of
$\widehat{\mathcal F}$ form the fibers of a smooth fiber bundle
$F_{O(q)}M \rightarrow \widehat{W}$, which is
$O(q)$-equivariant
\cite[Theorem 4.26(ii)]{Moerdijk-Mrcun (2003)}. Also, $W = \widehat{W}/O(q)$.
Let $\iota \: : \: \widehat{W} \rightarrow W$ denote the quotient map.

Let ${\mathcal T}$ be a {\em complete transversal} to ${\mathcal F}$
\cite[Example 5.19]{Moerdijk-Mrcun (2003)}.
Because $M$ is compact,
we can assume that
${\mathcal T}$ has a
finite number of connected components, each being the
interior of a smooth manifold-with-boundary.
Let ${\mathcal G}_{\mathcal T}$ be the corresponding
{\em \'etale holonomy groupoid} \cite[Example 5.19]{Moerdijk-Mrcun (2003)}.
Its space of units is ${\mathcal T}$.
Then ${\mathcal G}_{\mathcal T}$ is a complete Riemannian groupoid.
Its weak equivalence class is independent of the choice of
complete transversal ${\mathcal T}$.

We write $d\mu_{\mathcal T}$ for the Riemannian density measure on ${\mathcal T}$.

\begin{example} \label{example1}
Let $\Gamma$ be a finitely presented discrete group. Let $G$ be a compact
Lie group which acts isometrically and effectively on a connected compact Riemannian manifold $Z$.
Suppose that $i : \Gamma \rightarrow G$ is an injective homomorphism. Suppose that $Y$ is a
connected compact manifold with $\pi_1(Y, y_0) = \Gamma$. Let $c : \widetilde{Y}
\rightarrow Y$ be the universal cover.
Then $M = (\widetilde{Y} \times Z)/\Gamma$ has a Riemannian foliation whose leaves are
the images in $M$ of $\{\widetilde{Y} \times \{z\} \}_{z \in Z}$.
It is an example of a suspension foliation.
There is a complete
transversal $(c^{-1}(y_0) \times Z)/\Gamma \cong Z$. Then
${\mathcal G}_{\mathcal T}$ is the cross-product groupoid $Z \rtimes \Gamma$.
\end{example}

We will want to take the closure of ${\mathcal G}_{\mathcal T}$ in a certain sense,
following \cite{Haefliger (1985),Salem (1988a),Salem (1988b)}.
To do so, let
\begin{equation} \label{2.2}
J^1(\mathcal{T})=\left\{(p_1, p_2, A)\ |\ p_1, p_2 \in \mathcal{T}, A \in \Isom (T_{p_1} \mathcal{T}, T_{p_2}\mathcal{T}) \right\}
\end{equation}
be the groupoid of isometric $1$-jet elements, with the $1$-jet topology. It is a Lie groupoid in the
sense of \cite[Chapter 5.1]{Moerdijk-Mrcun (2003)}, but is not an \'etale groupoid unless $\dim({\mathcal T}) = 0$.
\begin{lemma} \label{lemma1}
$J^1(\mathcal{T})$ is a proper Lie groupoid in the sense of
\cite[Chapter 5.6]{Moerdijk-Mrcun (2003)}
\end{lemma}
\begin{proof}
The map $(s,r) \: : \: J^1(\mathcal{T})\to \mathcal{T} \times  \mathcal{T}$ sends
$(p_1, p_2, A)$ to $(p_1, p_2)$. It defines a fiber bundle with fibers diffeomorphic to the compact Lie group
$O(q)$. Hence it is a proper map.
\end{proof}

There is a homomorphism of ${\mathcal G}_{\mathcal T}$ into $J^1(\mathcal{T})$ that
sends $g \in {\mathcal G}_{\mathcal T}$ to $(s(g), r(g), dg_{s(g)}) \in J^1(\mathcal{T})$.
This homomorphism is injective, as follows from the fact that ${\mathcal G}_{\mathcal T}$
is effective, along with the fact that if $I$ is an isometry of a Riemannian manifold such
that $I(p) = p$ and $dI_p = \Id$ then $I$ is the identity in a neighborhood of $p$.

Let $\overline{{\mathcal G}_{\mathcal T}}$ be the
closure of ${\mathcal G}_{\mathcal T}$ in $J^1(\mathcal{T})$.
It is a subgroupoid of $J^1(\mathcal{T})$, again with
unit space ${\mathcal{T}}$. (Note that ${\mathcal{T}}$ is a smooth
manifold in its own right. The fact that it is the interior of
a compact manifold-with-boundary will not enter until Subsection
\ref{subsection3.3}.)
Now $\overline{{\mathcal G}_{\mathcal T}}$ is a
smooth subgroupoid of
$J^1(\mathcal{T})$ and so inherits a Lie groupoid structure;
see \cite[Section 2]{Salem (1988b)} and
(\ref{2.5}) below.
Note that $dg_{s(g)} \: : \: T_{s(g)} \mathcal{T} \rightarrow T_{r(g)} \mathcal{T}$ can be defined for all
$g \in \overline{{\mathcal G}_{\mathcal T}}$.

\begin{lemma} \label{lemma2}
$\overline{{\mathcal G}_{\mathcal T}}$
is a proper Lie groupoid.
\end{lemma}
\begin{proof}
This follows from Lemma \ref{lemma1}, along with the fact that $\overline{{\mathcal G}_{\mathcal T}}$ is a
closed subset of $I(\mathcal{T})$.
\end{proof}

The orbit
space of $\overline{{\mathcal G}_{\mathcal T}}$ is $W$, the space of leaf
closures. Let $\sigma \: : \: {\mathcal T} \rightarrow W$ denote the quotient map.

\begin{example} \label{example2}
Continuing with Example \ref{example1}, suppose that the homomorphism $i : \Gamma \rightarrow G$ has dense image.
Then $\overline{{\mathcal G}_{\mathcal T}}$ is the cross-product groupoid $Z \rtimes G$.
\end{example}

In addition to its subspace topology from $J^1({\mathcal T})$, the groupoid
$\overline{{\mathcal G}_{\mathcal T}}$ has an \'etale topology, for which
$s$ and $r$ are local homeomorphisms.  In particular, each
$g \in \overline{{\mathcal G}_{\mathcal T}}$ has a
local extension to an isometry between neighborhoods of $s(g)$ and $r(g)$;
this follows from the fact that $g$ is a limit of elements of
${\mathcal G}_{\mathcal T}$ that have this property in a uniform way.
We will call this the {\em extendability} property of $g$. The
local extension of $g$ is given explicitly by $\exp_{r(g)} \circ dg_{s(g)} \circ
\exp_{s(g)}^{-1}$.
In what follows, when we refer to $\overline{{\mathcal G}_{\mathcal T}}$
we will give it the subspace topology, unless we say otherwise.

\begin{example} \label{example3}
Continuing with Example \ref{example2}, when we convert from the (proper) Lie groupoid topology on
$\overline{{\mathcal G}_{\mathcal T}}$ to the \'etale topology, the result is $Z \rtimes G_\delta$,
where $G_\delta$ denotes the discrete topology on $G$.
\end{example}

\subsection{Normal frame bundle} \label{subsection2.3}

Let $\pi : F_{O(q)}{\mathcal T} \rightarrow {\mathcal T}$ be the orthonormal frame bundle of ${\mathcal T}$.
Then $\overline{{\mathcal G}_{{\mathcal T}}}$ acts on $F_{O(q)}{\mathcal T}$
by saying that if $g \in \overline{{\mathcal G}_{{\mathcal T}}}$
and $f$ is an orthonormal frame at $r(g)$
then $f \cdot g$ is the frame $(dg_{s(g)})^{-1}(f)$ at $s(g)$.

Let $\widehat{{\mathcal G}_{{\mathcal T}}}$ be the cross-product groupoid
\begin{equation}
F_{O(q)}{\mathcal T} \rtimes \overline{{\mathcal G}_{{\mathcal T}}} \: = \:
\{ (f,g) \: : \: g \in \overline{{\mathcal G}_{{\mathcal T}}}, \: f \text{ an orthonormal
frame at } r(g) \}.
\end{equation}
It has unit space $F_{O(q)}{\mathcal T}$, with
$s(f,g) = f \cdot g$ and $r(f,g) = f$, and orbit space $\widehat{W}$, with the
quotient map ${\widehat{\sigma}} \: : \: F_{O(q)}{\mathcal T} \rightarrow {\widehat{W}}$ being a smooth
submersion
\cite[Theorem 4.2]{Salem (1988a)}.
With abuse of terminology, we may call the subsets
${\widehat{\sigma}}^{-1}(\widehat{w})$
fibers.
There is a commutative diagram
\begin{alignat}{3} \label{2.3}
F_{O(q)}&{\mathcal T}     & &\stackrel{\widehat{\sigma}}{\longrightarrow}     & & \widehat{W} \\
\pi &\downarrow           & &                                             & &\downarrow \iota \notag \\
&{\mathcal T}             & &\stackrel{{\sigma}}{\longrightarrow}             & & W, \notag
\end{alignat}
where $\widehat{\sigma}$ is $O(q)$-equivariant, and $\pi$ and $\iota$ are the results of taking
$O(q)$-quotients.

The groupoid $\widehat{{\mathcal G}_{{\mathcal T}}}$ can be considered as a lift of
$\overline{{\mathcal G}_{{\mathcal T}}}$ to $F_{O(q)}{\mathcal T}$. It has trivial isotropy groups
and comes from the equivalence relation on $F_{O(q)}{\mathcal T}$ given by saying that
$f \sim f^\prime$ if and only if ${\widehat{\sigma}}(f) = {\widehat{\sigma}}(f^\prime)$.
There is an $O(q)$-equivariant isomorphism
\begin{equation} \label{2.4}
\widehat{{\mathcal G}_{ {\mathcal T}}} =
\left( F_{O(q)}{\mathcal T} \times_{\widehat{W}} F_{O(q)}{\mathcal T}\right).
\end{equation}
Hence
\begin{equation} \label{2.5}
\overline{{\mathcal G}_{ {\mathcal T}}} =
\left( F_{O(q)}{\mathcal T} \times_{\widehat{W}} F_{O(q)}{\mathcal T}\right)/O(q)
\end{equation}
as Lie groupoids.

The groupoid $\widehat{{\mathcal G}_{ {\mathcal T}}}$ also has an \'etale structure, coming from
that of 
$\overline{{\mathcal G}_{ {\mathcal T}}}$.
To see this in terms of local diffeomorphisms,
given $\widehat{g} \in \widehat{{\mathcal G}_{ {\mathcal T}}}$, write it as a pair
$(f,g)$ with $g \in \overline{{\mathcal G}_{{\mathcal T}}}$ and $f$ an orthonormal frame at $r(g)$.
Let $L \: : \: U \rightarrow V$ be an extension of $g$ to an isometry, where $U$ is a neighborhood of
$s(g) \in {\mathcal T}$ and $V$ is a neighborhood of $r(g) \in {\mathcal T}$.
Then the lift $\widehat{L} \: : \: F_{O(q)} U \rightarrow F_{O(q)} V$ is a diffeomorphism from a
neighborhood of $s(\widehat{g}) \in F_{O(q)}{\mathcal T}$ to a neighborhood of
$t(\widehat{g}) \in F_{O(q)}{\mathcal T}$.

In particular,
\begin{equation} \label{2.6}
d\widehat{g}_{s(\widehat{g})} \: : \:
T_{s(\widehat{g})}F_{O(q)}{\mathcal T}
\rightarrow T_{r(\widehat{g})}F_{O(q)}{\mathcal T}
\end{equation}\
is well-defined.

There is a transverse Levi-Civita connection on $F_{O(q)}{\mathcal T}$, by means of which
one can construct a canonical {\em parallelism} of $F_{O(q)}{\mathcal T}$, i.e. vector fields $\{V^i\}$
that are pointwise linearly independent
and span $TF_{O(q)}{\mathcal T}$
\cite[Chapter 4.2.2]{Moerdijk-Mrcun (2003)}. This parallel structure is $\widehat{{\mathcal G}_{ {\mathcal T}}}$-invariant
in the sense that for all $\widehat{g} \in \widehat{{\mathcal G}_{ {\mathcal T}}}$,
$d\widehat{g}_{s(\widehat{g})} (V^i_{s(\widehat{g})}) = V^i_{r(\widehat{g})}$.

There is also a canonical
Riemannian metric on $F_{O(q)}{\mathcal T}$, which comes from saying that the vector fields
$\{V^i\}$ are pointwise orthonormal.  With respect to this Riemannian metric, the vertical $O(q)$-directions
are orthogonal to the horizontal directions (coming from the transverse Levi-Civita connection),
the $O(q)$-fibers are all isometric to the standard $O(q)$ with the bi-invariant Riemannian metric
of total volume one,
and the horizontal planes are isometric to their projections to ${\mathcal T}$.

With this Riemannian metric on $F_{O(q)}{\mathcal T}$, the submersion
${\widehat{\sigma}} \: : \: F_{O(q)} {\mathcal T} \rightarrow \widehat{W}$
becomes a Riemannian submersion.

Finally, we note that if ${\mathcal F}$ is transversely oriented then the above statements have
analogs in which $O(q)$ is replaced by $SO(q)$. Similarly, if ${\mathcal F}$ has a transverse spin
structure then the statements have analogs in which $O(q)$ is replaced by $\Spin(q)$.

\subsection{Haar system} \label{subsection2.4}

For $f \in F_{O(q)}{\mathcal T}$, let $d\mu^f$ be the measure on
$\widehat{{\mathcal G}_{\mathcal T}}$ which is supported on
$\widehat{{\mathcal G}_{\mathcal T}}^f \cong {\widehat{\sigma}}^{-1}({\widehat{\sigma}}(f))$
and is given there by the
fiberwise Riemannian density.

To define the mean curvature form
$\widehat{\tau} \in \Omega^1( F_{O(q)} {\mathcal T})$ of the
fibers, choose $f \in F_{O(q)}{\mathcal T}$. Given a vector
$\widehat{X}_f \in T_fF_{O(q)}{\mathcal T}$, extend it to a vector field
$\widehat{X}$ on 
$\widehat{{\mathcal G}_{\mathcal T}}^f \cong {\widehat{\sigma}}^{-1}({\widehat{\sigma}}(f))$,
the 
$\widehat{{\mathcal G}_{{\mathcal T}}}$-orbit
of $f$,
so that for all 
$\widehat{g} \in
\widehat{{\mathcal G}_{{\mathcal T}}}^f$
we have
$d\widehat{g}_{s(\widehat{g})}(\widehat{X}_{s(\widehat{g})}) = \widehat{X}_f$. We can
find $\epsilon > 0$ and a small neighborhood $U$ of $f$ in ${\widehat{\sigma}}^{-1}({\widehat{\sigma}}(f))$ so that
the geodesic flow $\phi_t(f^\prime) = \exp_{f^\prime}(t\widehat{X}_{f^\prime})$ is defined for
all $t \in (-\epsilon,\epsilon)$
and $f^\prime \in U$, and $\phi_t$ maps $U$ diffeomorphically to its image
in a fiber ${\widehat{\sigma}}^{-1}(\gamma(t))$. Here $\gamma$ is the geodesic on $\widehat{W}$
starting from ${\widehat{\sigma}}(f)$, with initial vector $d{\widehat{\sigma}}_f(\widehat{X}_f)$.
Define the Lie derivative
\begin{equation} \label{2.7}
({\mathcal L}_{\widehat{X}} d\mu)^f \: = \:
\frac{d}{dt} \Big|_{t=0} \phi_t^* d\mu^{\phi_t(f)}
\end{equation}
Then
\begin{equation} \label{2.8}
\widehat{\tau}(\widehat{X}_f) = \frac{
({\mathcal L}_{\widehat{X}} d\mu)^f
}{
d\mu^f
} \Big|_f.
\end{equation}

\begin{lemma} \label{lemma3}
$\widehat{\tau}$ is a closed $1$-form which is
$\widehat{{\mathcal G}_{\mathcal T}}$-basic and $O(q)$-basic.
\end{lemma}
\begin{proof}
The form $\widehat{\tau}$ is clearly
$\widehat{{\mathcal G}_{\mathcal T}}$-invariant and $O(q)$-invariant.
As $\widehat{{\mathcal G}_{\mathcal T}}$ and $O(q)$ act on
$F_{O(q)}{\mathcal T}$ isometrically, if $\widehat{X}_f \in
T_f F_{O(q)}{\mathcal T}$ is tangent to the $\widehat{{\mathcal G}_{\mathcal T}}$-orbit
of $f$, or the $O(q)$-orbit of $f$, then $({\mathcal L}_{\widehat{X}} d\mu)^f = 0$.
Hence $\widehat{\tau}$ is
$\widehat{{\mathcal G}_{\mathcal T}}$-basic and $O(q)$-basic.

To see that $\widehat{\tau}$ is closed, we will define a smooth positive function $\widehat{F}$ in
a neighborhood $N$ of $f$ so that $\widehat{\tau} = d \log \widehat{F}$ there.
(The neighborhood $N$ will be taken small enough so that the following
construction makes sense.)
For $f^\prime \in N$,
we write ${\widehat{\sigma}}(f^\prime) = \exp_{{\widehat{\sigma}}(f)} \widehat{V}$ for a unique
$\widehat{V} \in T_{{\widehat{\sigma}}(f)} \widehat{W}$. Let $\widehat{X}$ be the horizontal lift
of $\widehat{V}$ to ${\widehat{\sigma}}^{-1}({\widehat{\sigma}}(f))$. 
For $f^{\prime \prime} \in \widehat{\sigma}^{-1}(\widehat{\sigma}(f))$,
put
$\phi_1(f^{\prime \prime}) = \exp_{f^{\prime \prime}} 
\widehat{X}_{f^{\prime \prime}}$. Put
\begin{equation} \label{2.9}
\widehat{F}(f^\prime) = \frac{
d\mu^{f^\prime}
}{
(\phi_1^*)^{-1} d\mu^f
} \Big|_{f^\prime}.
\end{equation}
This defines $\widehat{F}$ on $N$ so that $\widehat{\tau} = d \log \widehat{F}$ on $N$.
Hence $\widehat{\tau}$ is closed.
\end{proof}

\begin{corollary} \label{corollary2}
Let $\tau \in \Omega^1( {\mathcal T})$ be the unique $1$-form
such that $\widehat{\tau}= \pi^* \tau$.
Then $\tau$ is closed and $\overline{{\mathcal G}_{\mathcal T}}$-basic.
\end{corollary}

Recall the notion of a Haar system for a Lie groupoid; see,
for example, \cite[Definition 1.1]{Tu (2000)}.
Now $\{d\mu^f\}_{f \in F_{O(q)}{\mathcal T}}$ is a Haar system for
$\widehat{{\mathcal G}_{\mathcal T}}$. In particular, $d\mu^f$ is a measure on
$\widehat{{\mathcal G}_{\mathcal T}}$ whose suppport is 
$\widehat{\mathcal G}_{\mathcal T}^f$,
and the family of measures $\{d\mu^f\}_{f \in F_{O(q)}{\mathcal T}}$
is $\widehat{{\mathcal G}_{\mathcal T}}$-invariant in an appropriate sense.

Given $p \in {\mathcal T}$, choose $f \in F_{O(q)}{\mathcal T}$ so
that $\pi(f) = p$. There is a diffeomorphism $i_{p,f} :
\overline{{\mathcal G}_{\mathcal T}}^p \rightarrow
\widehat{{\mathcal G}_{\mathcal T}}^f$ given by
$i_{p,f}(g) = (f,g)$. Let $d\mu^p$ be the measure on
$\overline{{\mathcal G}_{\mathcal T}}$ which is supported on
$\overline{{\mathcal G}_{\mathcal T}}^p$ and is given there by
$i_{p,f}^* d\mu^f$, where we think of $d\mu^f$ as a density
on $\widehat{{\mathcal G}_{\mathcal T}}^f$.
Then $d\mu^p$ is independent of the choice of $f$,
as follows from the fact that
the family $\{d\mu^f\}_{p \in {\mathcal T}}$ is $O(q)$-equivariant.
One can check that
$\{d\mu^p\}_{p \in {\mathcal T}}$ is a Haar system for
$\overline{{\mathcal G}_{\mathcal T}}$.

\begin{example} \label{example4}
Continuing with Example \ref{example3}, given $p \in Z$, the measure $d\mu^p$ on
$\overline{{\mathcal G}_{\mathcal T}}^p \cong G$ can be described as
follows.  Let $\{e_i\}$ be a basis for ${\mathfrak g}$ such that the normalized
Haar measure on $G$ is $d\mu_G = \wedge_i e_i^*$. Let $\{V_i\}$ be
the corresponding vector fields on $Z$.
The action of $V_i$ on $F_{O(q)}Z$ breaks up as
$V^i \oplus \nabla V^i$, with respect to the decomposition
$T F_{O(q)}Z \: = \: \pi^* TZ \oplus TO(q)$ of $TF_{O(q)}Z$ into its
horizontal and vertical subbundles.  (Note that because $V^i$ is
Killing, $\nabla V^i$ is a skew-symmetric $2$-tensor.)
Put
\begin{equation} \label{2.10}
M_{ij}(p) \: = \: \langle V_i(p), V_j(p) \rangle \: + \:
\langle \nabla V_i(p), \nabla V_j(p) \rangle.
\end{equation}
Note that the matrix $M(p)$ is positive-definite.
Then $d\mu^p \: = \: \sqrt{\det(M(p))} \: d\mu_G$.
\end{example}

We now construct a cutoff function for
$\overline{{\mathcal G}_{\mathcal T}}$.

\begin{lemma} \label{lemma4}
There is a nonnegative cutoff function
$\phi \in C^\infty_c({\mathcal T})$ for $\overline{{\mathcal G}_{\mathcal T}}$,
meaning that for all $p \in {\mathcal T}$,
\begin{equation} \label{2.11}
\int_{\overline{{\mathcal G}_{\mathcal T}}^p} \phi^2(s(g)) \: d\mu^p(g) \: = \: 1.
\end{equation}
\end{lemma}
\begin{proof}
The proof is similar to that of
\cite[Proposition 6.11]{Tu (1999)}. The difference is that we use
$\phi^2$ instead of $\phi$ as in \cite[Proposition 6.11]{Tu (1999)}.
Choose any nonnegative
$\psi  \in C_c^\infty({\mathcal T})$ such that $\int_{\overline{{\mathcal G}_{\mathcal T}}^p} \psi^2(s(g)) \: d\mu^p(g)>0$ for all $p \in {\mathcal T}$. (Such a $\psi$ exists
because the orbit space of $\overline{{\mathcal G}_{\mathcal T}}$ is compact).
Then set
\begin{equation}  \label{2.12}
\phi = \frac{\psi }{\sqrt{\int_{\overline{{\mathcal G}_{\mathcal T}}^p}
\psi^2(s(g)) \: d\mu^p(g)}}.
\end{equation}
\end{proof}

\begin{example} \label{example5}
Continuing with Example \ref{example4}, we can take
$\phi^2(p) \: = \: \frac{1}{\sqrt{\det(M(p))}}$.
\end{example}

\subsection{The Lie algebroid of the groupoid closure} \label{subsection2.5}

Molino theory is phrased as a structure on the foliated manifold $M$ in
\cite[Chapter 4]{Moerdijk-Mrcun (2003)} and \cite{Molino (1988)}, and as a structure on the
transversal ${\mathcal T}$ in \cite{Haefliger (1985),Salem (1988a),Salem (1988b)}.
The relationship between them is that the structure on $M$ pulls back from the structure on ${\mathcal T}$
\cite[Section 3.4]{Salem (1988a)}.

Let $\overline{{\mathfrak g}_{{\mathcal T}}}$ be the Lie algebroid of
$\overline{{\mathcal G}_{{\mathcal T}}}$, as defined in \cite[Chapter 6]{Moerdijk-Mrcun (2003)}.
Then $\overline{{\mathfrak g}_{{\mathcal T}}}$ is a $\overline{{\mathcal G}_{{\mathcal T}}}$-equivariant flat
vector bundle over
${\mathcal T}$ whose fibers are copies of a fixed Lie algebra ${\mathfrak g}$.
(The flat connection on $\overline{{\mathfrak g}_{{\mathcal T}}}$ is related to the
extendability of elements of $\overline{{\mathcal G}_{{\mathcal T}}}$.)
The holonomy of the
flat connection on $\overline{{\mathfrak g}_{{\mathcal T}}}$
lies in $\Aut({\mathfrak g})$.
If $P : (U \times {\mathfrak g}) \rightarrow \overline{{\mathfrak g}_{{\mathcal T}}}$ is a
local parallelization of $\overline{{\mathfrak g}_{{\mathcal T}}}$
and $an : \overline{{\mathfrak g}_{{\mathcal T}}} \rightarrow T{\mathcal T}$ is the anchor map
then $an \circ P$ describes a Lie algebra of Killing vector fields on $U$, isomorphic to ${\mathfrak g}$.

The pullback $\pi^* \overline{{\mathfrak g}_{{\mathcal T}}}$ of
$\overline{{\mathfrak g}_{{\mathcal T}}}$ to $F_{O(q)}{\mathcal T}$
is isomorphic to the vertical tangent bundle $T^V F_{O(q)}{\mathcal T}$
of the submersion ${\widehat{\sigma}} \: : \: F_{O(q)}{\mathcal T} \rightarrow \widehat{W}$.

If $M$ is simply-connected then ${\mathfrak g}$ is abelian and
$\overline{{\mathfrak g}_{{\mathcal T}}} \: = \: {\mathcal T} \times {\mathfrak g}$;
see, for example, \cite{Haefliger-Salem (1988)}.

\begin{example} \label{example6}
Continuing with Example \ref{example5}, let ${\mathfrak g}$ be the Lie algebra of $G$.
Then $\overline{{\mathfrak g}_{{\mathcal T}}}$ is the product bundle $Z \times {\mathfrak g}$,
whose flat connection has trivial holonomy. The corresponding vector fields
on ${\mathcal T} = Z$ come from the $G$-action.
\end{example}

\begin{example} \label{example7}
Let $G$ be a finite-dimensional connected Lie group. Let ${\mathfrak g}$ be its Lie algebra.
Give $G$ a right-invariant Riemannian metric.
Let $\Gamma$ be a finite-presented discrete group.  Let $\Gamma \rightarrow G$ be an
injective homomorphism with dense image.  Let
$Y$ be a connected compact manifold with $\pi_1(Y, y_0) = \Gamma$. Let
$\widetilde{Y}$ be the universal cover.  
Suppose that $h : \widetilde{Y} \rightarrow G$ is
a $\Gamma$-equivariant fiber bundle, where $\Gamma$ acts on the right on $G$.

Then $Y$ has a Riemannian foliation ${\mathcal F}$ whose
leaves are the images, in $Y$, of the
connected components of the fibers of $h$. The foliation has dense leaves
and is transversally parallelizable.  Conversely, any Riemannian foliation on
a connected compact manifold, which has dense leaves and is transversally parallelizable,
arises from this construction \cite[Theorem 4.24]{Moerdijk-Mrcun (2003)}.

A transversal ${\mathcal T}$ to ${\mathcal F}$ can be formed by taking
appropriate local sections $U_i \rightarrow \widetilde{Y}$ of $h$.
Then
$\overline{{\mathfrak g}_{{\mathcal T}}}$ is the product bundle ${\mathcal T} \times {\mathfrak g}$,
whose flat connection has trivial holonomy. The corresponding vector fields
on ${\mathcal T} \cong \coprod_i U_i$ are the restrictions of the left-invariant vector
fields on $G$.

Note that in this construction, ${\mathfrak g}$ could be any finite-dimensional Lie algebra.
\end{example}

\subsection{Local transverse structure of a Riemannian foliation} \label{subsection2.6}

We describe the local transverse structure of a Riemannian foliation, following
\cite{Haefliger (1985),Haefliger (1988)}.

Fix $p \in {\mathcal T}$. Let $K$ denote the isotropy group at $p$ for
$\overline{{\mathcal G}_{{\mathcal T}}}$. Let
${\mathfrak k}$ denote the Lie algebra of $K$.
There is an injection $i \: : \: {\mathfrak k} \rightarrow {\mathfrak g}$.
Also, there is a representation $\ad \: : \: K \rightarrow \Aut({\mathfrak g})$
so that \\
1. $\ad \big|_{\mathfrak k}$
is the adjoint representation of $K$ on ${\mathfrak k}$. \\
2. $d\ad_{e}$ is the adjoint representation of ${\mathfrak k}$ on
${\mathfrak g}$, as defined using $i$.

Let $O_p$ be the $\overline{{\mathcal G}_{{\mathcal T}}}$-orbit
of $p$. Its tangent space $T_p O_p$ at $p$ is
isomorphic to ${\mathfrak g}/{\mathfrak k}$.
Put $V \: = \:
(T_p O_p)^\perp \subset T_p {\mathcal T}$. A slice-type theorem gives a representation
$\rho \: : \: K \rightarrow \Aut(V)$ with the property that
$\ad \oplus \rho \: : \: K \rightarrow \Aut(({\mathfrak g}/{\mathfrak k}) \oplus V)$ is
injective.

The quintuple $({\mathfrak g}, K, i, \ad, \rho)$ determines the weak
equivalence class of the
restriction of $\overline{{\mathcal G}_{\mathcal T}}$ (with the \'etale topology) to a small
invariant neighborhood of the orbit $O_p$.

Given such a quintuple, one can construct an explicit local model for
the transverse structure.
We will restrict here to the case when ${\mathfrak g}$ is solvable.
Then there is a Lie group $G$ with Lie algebra ${\mathfrak g}$, containing
$K$ as a subgroup, such that the
restriction of $\overline{{\mathcal G}_{\mathcal T}}$ to
a small invariant neighborhood of the
orbit $O_p$ is weakly equivalent to the cross-product groupoid
$(B(V) \times_K G) \rtimes G_\delta$, where $B(V)$ is a metric ball in $V$.

Finally, define a normal orbit type to be a quintuple $({\mathfrak g}, K, i, \ad, \rho)$
such that the invariant subspace $V^K$ vanishes.  Given a point $p \in {\mathcal T}$
and its associated quintuple $({\mathfrak g}, K, i, \ad, \rho)$, one obtains its
normal orbit type from replacing $V$ by $V/V^K$. There is a natural equivalence
relation on the set of possible normal orbit types.  Then there is a stratification
of ${\mathcal T}$, where each stratum is associated to a given equivalence class of
normal orbit types \cite[Section 3.3]{Haefliger (1985)}.

\subsection{The dualizing sheaf} \label{subsection2.7}

Let ${\mathcal O}_{\mathcal T}$ be the orientation bundle of ${\mathcal T}$.
It is a flat real line bundle on ${\mathcal T}$.
Put ${\mathcal D}_{{\mathcal T}} = \Lambda^{max} \overline{{\mathfrak g}_{{\mathcal T}}}
\otimes {\mathcal O}_{\mathcal T}$. It is
a $\overline{{\mathcal G}_{{\mathcal T}}}$-equivariant
flat real line bundle on ${\mathcal T}$.

The Haar system $\{d\mu^f\}_{f \in F_{O(q)} {\mathcal T}}$ gives a nowhere-zero
$O(q)$-invariant section of the pullback bundle
$\pi^* \Lambda^{max} \overline{{\mathfrak g}_{{\mathcal T}}}
 \cong \Lambda^{max} T^V F_{O(q)}{\mathcal T}$ on $F_{O(q)} {\mathcal T}$.
Tensoring with this section gives an $O(q)$-equivariant isomorphism $\widehat{\mathcal I} \: : \:
\Omega^*(F_{O(q)} {\mathcal T}; \pi^* {\mathcal O}_{{\mathcal T}}) \rightarrow
\Omega^*(F_{O(q)} {\mathcal T}; \pi^* {\mathcal D}_{{\mathcal T}})$.
This isomorphism descends to an isomorphism
${\mathcal I} \: : \: \Omega^*({\mathcal T}; {\mathcal O}_{\mathcal T}) \rightarrow
\Omega^*({\mathcal T}; {\mathcal D}_{\mathcal T})$.

\begin{lemma} \label{lemma5}
${\mathcal I}^{-1} \circ d \circ {\mathcal I} \: = \:
d - \tau \wedge$ on $\Omega^*({\mathcal T}; {\mathcal O}_{\mathcal T})$.
\end{lemma}
\begin{proof}
This follows from the local description of $\widehat{\tau} = \pi^* \tau$ as
$d \log \widehat{F}$ in the proof of Lemma \ref{lemma3}.
\end{proof}

Let $\HH^*_{\inv}({\mathcal T})$ be the cohomology of the
$\overline{{\mathcal G}_{{\mathcal T}}}$-invariant differential
forms on ${\mathcal T}$, and similarly for
$\HH^*_{\inv}({\mathcal T}; {\mathcal D}_{\mathcal T})$.
Then $\HH^*_{\inv}({\mathcal T})$ is isomorphic to the
basic cohomology $\HH^*_{\bas}(M)$ of the foliated manifold $M$,
which is invariant under foliated homeomorphisms
\cite{El-Kacimi2 (1993)}.
Also,
$\HH^*_{\inv}({\mathcal T}; {\mathcal D}_{{\mathcal T}})$ is isomorphic
to $\HH^*_{\bas}(M; {\mathcal D}_M)$, where ${\mathcal D}_M$ is the
pullback of ${\mathcal D}_{\mathcal T}$ from ${\mathcal T}$ to $M$.
From \cite{Sergiescu (1985)}, for all $0 \le i \le \dim({\mathcal T})$,
there is a nondegenerate pairing
\begin{equation} \label{2.13}
\HH^i_{\inv}({\mathcal T}) \times
\HH^{\dim({\mathcal T}) -i}_{\inv}({\mathcal T}; {\mathcal D}_{{\mathcal T}}) \rightarrow \R.
\end{equation}
More generally,
if $E$ is a $\overline{{\mathcal G}_{{\mathcal T}}}$-equivariant flat real
vector bundle on ${\mathcal T}$ then there is a nondegenerate pairing
\begin{equation} \label{2.14}
\HH^i_{\inv}({\mathcal T}; E) \times
\HH^{\dim({\mathcal T}) -i}_{\inv}({\mathcal T}; E^* \otimes
{\mathcal D}_{{\mathcal T}}) \rightarrow \R.
\end{equation}

The closed $1$-form $\tau$ itself defines a class
$[\tau] \in \HH^1_{\inv}({\mathcal T})$.

If ${\mathcal D}_{{\mathcal T}}$ is topologically trivial, as a
$\overline{{\mathcal G}_{{\mathcal T}}}$-equivariant real
line bundle on ${\mathcal T}$, then we can take the (positive)
square root of its holonomies to obtain ${\mathcal D}_{{\mathcal T}}^\frac12$,
a $\overline{{\mathcal G}_{{\mathcal T}}}$-equivariant
flat real line bundle. We obtain a nondegenerate bilinear form
on $\HH^*_{\inv}({\mathcal T}; {\mathcal D}_{{\mathcal T}}^{\frac12})$ from (\ref{2.14}).
Hence if $\dim({\mathcal T})$ is divisible by four then the basic signature
${{\sigma}}(M, {\mathcal F}; {\mathcal D}_{{\mathcal T}}^{\frac12})$
can be defined to be the index of the quadratic form on
$\HH^{\dim({\mathcal T})/2}_{\inv}({\mathcal T};
{\mathcal D}_{{\mathcal T}}^{\frac12})$.
Note that $\HH^*_{\inv}({\mathcal T}; {\mathcal D}_{{\mathcal T}}^{\frac12})$
is isomorphic to the cohomology of $d \: - \: \frac12 \: \tau \wedge$
on $\Omega^*({\mathcal T})$. If in addition $[\tau] = 0$ then we can write
$\tau = dH$ for some $H \in C^\infty_{\inv}({\mathcal T})$, so
$d \: - \: \frac12 \: \tau \wedge \: = \: e^{H/2} \circ d \circ e^{-H/2}$
is conjugate to $d$ on $\Omega^*({\mathcal T})$.

Similarly, we can define a basic Euler characteristic
${{\chi}}(M, {\mathcal F}; {\mathcal D}_{{\mathcal T}}^{\frac12})$.

\section{Transverse Dirac-type operators} \label{section3}

In this section we construct the basic Dirac-type operator.
Subsection \ref{subsection3.1} relates transverse differentiation with
groupoid integration.  In Subsection \ref{subsection3.2} we define a
map $\alpha$ from holonomy-invariant sections of the
transverse Clifford module to non-invariant sections, and
a map $\beta$ which goes the other way.  We show that
$\beta \circ \alpha = \Id$ and $\beta = \alpha^*$. A
projection operator is then defined by $P = \alpha \circ
\beta$. It comes from the action of an idempotent in the
groupoid algebra. The invariant Dirac-type operator $D_{\inv}$ is the
compression of the transverse Dirac-type operator $D_{\APS}$ by
$P$. In Subsection \ref{subsection3.3}, we write $D_{\inv}$ explicitly as a differential operator.

\subsection{Transverse differentiation} \label{subsection3.1}

Let $E$ be a $\overline{{\mathcal G}_{\mathcal T}}$-equivariant vector bundle on ${\mathcal T}$.
Given $g \in \overline{{\mathcal G}_{\mathcal T}}$ and $e \in E_{s(g)}$, let
$e \cdot g^{-1} \in E_{r(g)}$ denote the action of $g^{-1}$ on $e$.
Given a compactly-supported element
$\xi \in C^\infty_c(\mathcal{T}; E)$, with a slight abuse of notation we write
\begin{equation} \label{3.1}
\int_{\overline{{\mathcal G}_{\mathcal T}}^p}
  \xi_{s(g)} \cdot g^{-1} \: d\mu^p(g)
\end{equation}
for the element of $C^\infty(\mathcal{T}; E)$ whose value at $p \in {\mathcal T}$ is given by (\ref{3.1}).

\begin{lemma} \label{lemma6}
We have an identity in $\Omega^1(\mathcal{T}; E)$ :
\begin{equation} \label{3.2}
\nabla^E \int_{\overline{{\mathcal G}_{\mathcal T}}^p}
  \xi_{s(g)} \cdot g^{-1} \: d\mu^p(g)\: = \:
\int_{\overline{{\mathcal G}_{\mathcal T}}^p}
    (\nabla^E \xi)_{s(g)} \cdot g^{-1} \: d\mu^p(g) \: + \: \tau_p \: \int_{\overline{{\mathcal G}_{\mathcal T}}^p}
  \xi_{s(g)} \cdot g^{-1} \: d\mu^p(g).
\end{equation}
\end{lemma}
\begin{proof}
Put $\widehat{\nabla}^E \: = \: \pi^*\nabla^E$ and $\widehat{\xi} \: = \: \pi^*\xi$.
Choose $f \in F_{O(q)} \mathcal{T}$ so that  $\pi(f)=p$.
Given a vector
$\widehat{X}_f \in T_fF_{O(q)}{\mathcal T}$, extend it to a vector field
$\widehat{X}$ on ${\widehat{\sigma}}^{-1}({\widehat{\sigma}}(f))$,
the 
$\widehat{{\mathcal G}_{{\mathcal T}}}$-orbit
of $f$,
so that for all 
$\widehat{g} \in
\widehat{{\mathcal G}_{{\mathcal T}}}^f$,
we have
$d\widehat{g}_{s(\widehat{g})}(\widehat{X}_{s(\widehat{g})}) = \widehat{X}_f$.
By the $\widehat{{\mathcal G}_{{\mathcal T}}}$-invariance of $\widehat{\nabla}^E$,
\begin{align} \label{3.3}
\widehat{\nabla}_{\widehat{X}}^E \int_{\widehat{{\mathcal G}_{\mathcal T}}^f}
\widehat{\xi}_{s(\widehat{g})} \cdot \widehat{g}^{-1}  \: d\mu^f(\widehat{g}) \: & = \:
 \int_{\widehat{{\mathcal G}_{\mathcal T}}^f}
(\widehat{\nabla}^E_{\widehat{X}} \widehat{\xi})_{s(\widehat{g})} \cdot
\widehat{g}^{-1} \: d\mu^f(\widehat{g}) \: + \:
\int_{\widehat{{\mathcal G}_{\mathcal T}}^f}
\widehat{\xi}_{s(\widehat{g})} \cdot \widehat{g}^{-1} \: \mathcal{L}_{\widehat{X}} d\mu^f(\widehat{g}) \\
& = \: \int_{\widehat{{\mathcal G}_{\mathcal T}}^f}
(\widehat{\nabla}^E_{\widehat{X}} \widehat{\xi})_{s(\widehat{g})} \cdot \widehat{g}^{-1} \:
d\mu^f(\widehat{g}) \: + \:
\int_{\widehat{{\mathcal G}_{\mathcal T}}^f}
\widehat{\xi}_{s(\widehat{g})} \cdot \widehat{g}^{-1} \:
\widehat{\tau}(\widehat{X})_{s(\widehat{g})} d\mu^f(\widehat{g}). \notag
\end{align}
Since $\widehat{\tau}(\widehat{X})_{s(g)} \: = \:
\widehat{\tau}(\widehat{X})_f$, the lemma follows.
\end{proof}

\begin{corollary} \label{corollary3}
If $\omega \in \Omega^*_c(\mathcal{T})$ then
\begin{equation} \label{3.4}
d \int_{\overline{{\mathcal G}_{\mathcal T}}^p}
\omega_{s(g)} \cdot g^{-1} \:  d\mu^p(g) \: = \:
\int_{\overline{{\mathcal G}_{\mathcal T}}^p}
(d \omega)_{s(g)} \cdot g^{-1} \: d\mu^p(g) \: + \: \tau_p \wedge \int_{\overline{{\mathcal G}_{\mathcal T}}^p}
\omega_{s(g)} \cdot g^{-1} \: d\mu^p(g).
\end{equation}
\end{corollary}

Suppose now that ${\mathcal E}$ is a $\overline{{\mathcal G}_{\mathcal T}}$-equivariant Clifford module
on ${\mathcal T}$. In particular, if
$X \in T_p{\mathcal T}$ then the Clifford action of $X$ is an operator $c(X) \in \End({\mathcal E}_p)$ with
$c(X)^2 \: = \: - \: |X|^2$.
Let $D$ be the Dirac-type operator on
$C^\infty_c({\mathcal T}; {\mathcal E})$. It is
a symmetric operator.
\begin{corollary} \label{corollary4}
If $\xi \in C^\infty_c(\mathcal{T}; \mathcal{E})$ then
\begin{equation} \label{3.5}
D \int_{\overline{{\mathcal G}_{\mathcal T}}^p}
  \xi_{s(g)} \cdot g^{-1} \: d\mu^p(g) \: = \:
\int_{\overline{{\mathcal G}_{\mathcal T}}^p}
    (D \xi)_{s(g)} \cdot g^{-1} \: d\mu^p(g) \: + \: c(\tau_p) \int_{\overline{{\mathcal G}_{\mathcal T}}^p}
  \xi_{s(g)} \cdot g^{-1} \: d\mu^p(g),
\end{equation}
where we have identified $\tau_p$ with its dual vector.
\end{corollary}

\subsection{A projection operator} \label{subsection3.2}

Recall the cutoff function $\phi$ from Lemma \ref{lemma4}.
Let $\left( L^2({\mathcal T}; \mathcal{E}) \right)^{\overline{{\mathcal G}_{\mathcal T}}}$ denote the
${\overline{{\mathcal G}_{\mathcal T}}}$-invariant elements of
$L^2({\mathcal T}; \mathcal{E})$.
Define maps $\alpha \colon \left( L^2({\mathcal T}; \mathcal{E}) \right)^{\overline{{\mathcal G}_{\mathcal T}}} \to
L^2({\mathcal T}; {\mathcal E}) $ and
$\beta\colon L^2({\mathcal T}; {\mathcal E}) \to \left( L^2({\mathcal T}; {\mathcal E}) \right)^{\overline{{\mathcal G}_{\mathcal T}}}$ by
\begin{equation} \label{3.6}
\alpha(\xi) \: = \: \phi \xi
\end{equation}
and
\begin{equation} \label{3.7}
(\beta (\eta))_p \: = \: \int_{g \in \overline{{\mathcal G}_{\mathcal T}}^p}  \eta_{s(g)} \cdot g^{-1}
\: \phi_{s(g)} \: d\mu^p(g).
\end{equation}

\begin{lemma} \label{lemma7}
We have $\beta \circ \alpha = \Id$.
\end{lemma}

\begin{proof}
If $\xi \in \left( L^2({\mathcal T}; {\mathcal E}) \right)^{\overline{{\mathcal G}_{\mathcal T}}}$ then
\begin{equation} \label{3.8}
(\beta ( \alpha (\xi)))_p \: =  \: \int_{g \in \overline{{\mathcal G}_{\mathcal T}}^p}  \xi_{s(g)} \cdot g^{-1}
\: \phi^2_{s(g)} \:  d\mu^p(g).
\end{equation}
Since $\xi $ is $\overline{{\mathcal G}_{\mathcal T}}$-invariant, $\xi_{s(g)} \cdot g^{-1} =\xi_p$ and so
\begin{equation} \label{3.9}
\int_{g \in \overline{{\mathcal G}_{\mathcal T}}^p}  \xi_{s(g)} \cdot g^{-1}
\: \phi^2_{s(g)} \:  d\mu^p(g) \: = \:
\xi_p \:
\int_{g \in \overline{{\mathcal G}_{\mathcal T}}^p} \phi^2_{s(g)} \:  d\mu^p(g) \: = \: \xi_p.
\end{equation}
This proves the lemma.
\end{proof}

It follows that $\alpha$ is injective and induces an isomorphism between
$\left( L^2({\mathcal T}; {\mathcal E})
\right)^{\overline{{\mathcal G}_{\mathcal T}}} $
and a subspace of $L^2({\mathcal T}; {\mathcal E})$.
We equip $\left( L^2({\mathcal T}; {\mathcal E})
\right)^{\overline{{\mathcal G}_{\mathcal T}}} $
with the inner product induced by this isomorphism.
Explicitly, for $\xi,\zeta \in \left( L^2({\mathcal T}; {\mathcal E}) \right)^{\overline{{\mathcal G}_{\mathcal T}}}$, we have
\begin{equation} \label{3.10}
 \langle \xi, \zeta \rangle \:  = \: \int_{\mathcal{T}} (\xi_p, \zeta_p) \: \phi^2(p) \: d\mu_{\mathcal{T}}(p),
\end{equation}
where $d\mu_{\mathcal{T}}$ is the Riemannian density on $\mathcal{T}$.
Note that this generally differs from the inner product on
$\left( L^2({\mathcal T}; {\mathcal E}) \right)^{\overline{{\mathcal G}_{\mathcal T}}} $ coming
from its embedding in $L^2({\mathcal T}; {\mathcal E})$.

We define a sheaf ${{\mathcal S}}_\infty$ on ${W}$ by saying that if
${U}$ is an open subset of ${W}$ then ${{\mathcal S}}_\infty({U}) \: = \:
(C^\infty(\sigma^{-1}(U); {\mathcal E}))^{\overline{{\mathcal G}_{{\mathcal T}}}}$. Similarly,
we define a sheaf ${\mathcal S}_2$ on $W$ by
${{\mathcal S}}_2({U}) \: = \:
(L^2(\sigma^{-1}(U); {\mathcal E}))^{\overline{{\mathcal G}_{{\mathcal T}}}}$.
The global sections
${\mathcal S}_2(W)$ are the same as
$\left( L^2({\mathcal T}; \mathcal{E}) \right)^{\overline{{\mathcal G}_{\mathcal T}}}$.

Let $d\mu_{\widehat{W}}$ denote the Riemannian measure on $\widehat{W}$.
Put $d{\mu}_{{W}} \: = \: \iota_* d\mu_{\widehat{W}}$, a measure on $W$.

Given
$\xi, \zeta \in \left( L^2({\mathcal T}; {\mathcal E}) \right)^{\overline{{\mathcal G}_{\mathcal T}}} $,
the pointwise inner product function $(\xi, \zeta)(p)$ pulls back under $\iota$ from a measurable function
on $W$, which we denote by $(\xi, \zeta)(w)$.

\begin{proposition} \label{proposition1}
We have
\begin{equation} \label{3.11}
\langle \xi, \zeta \rangle \: = \: \int_W (\xi, \zeta)(w) \: d\mu_W(w).
\end{equation}
\end{proposition}
\begin{proof}
Put $\widehat{\phi} \: = \: \pi^* \phi$.
Let $d\mu_{F_{O(q)}\mathcal{T}/\widehat{W}}$ denote the Riemannian densities on the preimages of ${\widehat{\sigma}}$.
Then
\begin{align} \label{3.12}
\int_{\mathcal{T}} (\xi_p, \zeta_p) \: \phi^2_p \:
d\mu_{\mathcal{T}}(p) \: & = \: \int_{F_{O(q)}\mathcal{T}}  ((\pi^*\xi)_f, (\pi^*\zeta)_f)
\: \widehat{\phi}^2_f \: d\mu_{F_{O(q)}\mathcal{T}}(f) \\
& = \: \int_{\widehat{W}}  (\pi^*\xi, \pi^*\zeta)(\widehat{w}) \: \left( \int_{F_{O(q)}\mathcal{T}/\widehat{W}} \widehat{\phi}^2_f \:
d\mu_{F_{O(q)}\mathcal{T}/\widehat{W}}(f) \right) d\mu_{\widehat{W}}(\widehat{w}) \notag \\
& = \: \int_{\widehat{W}}  (\pi^*\xi, \pi^*\zeta)(\widehat{w}) \: d\mu_{\widehat{W}}(\widehat{w}) \notag \\
& = \: \int_W (\xi, \zeta)(w) \: d\mu_W(w). \notag
\end{align}
This proves the proposition.
\end{proof}

\begin{corollary} \label{corollary5}
The inner product (\ref{3.10})
on $\left( L^2({\mathcal T}; {\mathcal E}) \right)^{\overline{{\mathcal G}_{\mathcal T}}}$
is independent of the choice of the cut-off function $\phi$.
\end{corollary}

We will denote $\left( L^2({\mathcal T}; {\mathcal E}) \right)^{\overline{{\mathcal G}_{\mathcal T}}} $,
equipped with the inner product (\ref{3.10}), by $L^2({\mathcal S}, d\mu_W)$.

\begin{proposition} \label{proposition2}
$\beta = \alpha^*$.
\end{proposition}
\begin{proof}
Choose
$\eta \in L^2({\mathcal T}; {\mathcal E})$ and
$\xi \in \left( L^2({\mathcal T}; {\mathcal E}) \right)^{\overline{{\mathcal G}_{\mathcal T}}} $.
Then
\begin{equation} \label{3.13}
\langle \beta \eta, \xi \rangle \: = \: \int_{F_{O(q)}\mathcal{T}} \int_{\widehat{{\mathcal G}_{\mathcal T}}^f} \widehat{\phi}_f^2
\: \widehat{\phi}_{s(\widehat{g})} \: \left( (\pi^*\eta)_{s(\widehat{g}))}\cdot \widehat{g}^{-1},
(\pi^*\xi)_f \right) \: d\mu^f(\widehat{g}) \: d\mu_{F_{O(q)}\mathcal{T}}(f).
\end{equation}
Using the $\overline{{\mathcal G}_{\mathcal T}}$-invariance of $\xi$,
\begin{align} \label{3.14}
& \int_{F_{O(q)}\mathcal{T}} \int_{\widehat{{\mathcal G}_{\mathcal T}}^f} \widehat{\phi}_f^2
\: \widehat{\phi}_{s(\widehat{g})} \: \left( (\pi^*\eta)_{s(\widehat{g}))}\cdot \widehat{g}^{-1},
(\pi^*\xi)_f \right) \: d\mu^f(\widehat{g}) \: d\mu_{F_{O(q)}\mathcal{T}}(f) \: = \\
& \int_{\widehat{{\mathcal G}_{\mathcal T}}} \widehat{\phi}_{r(\widehat{g})}^2 \: \widehat{\phi}_{s(\widehat{g})} \:
\left( (\pi^*\eta)_{s(\widehat{g})}, (\pi^*\xi)_{s(\widehat{g})} \right)  \:
d\mu_{\widehat{{\mathcal G}_{\mathcal T}} }(\widehat{g}), \notag
\end{align}
where $d\mu_{\widehat{{\mathcal G}_{\mathcal T}} }$ is  the measure on $\widehat{{\mathcal G}_{\mathcal T}}$ induced by the Haar system $\{d\mu^f\}_{f \in F_{O(q)}{\mathcal T}}$ and the Riemannian measure $d\mu_{F_{O(q)}\mathcal{T}}$.
Since $d\mu_{\widehat{{\mathcal G}_{\mathcal T}} }$ is invariant under the involution
$\widehat{g} \mapsto \widehat{g}^{-1}$
on $\widehat{{\mathcal G}_{\mathcal T}}$,
\begin{align} \label{3.15}
& \int_{\widehat{{\mathcal G}_{\mathcal T}}} \widehat{\phi}_{r(\widehat{g})}^2 \: \widehat{\phi}_{s(\widehat{g})} \:
\left( (\pi^*\eta)_{s(\widehat{g})}, (\pi^*\xi)_{s(\widehat{g})} \right)  \:
d\mu_{\widehat{{\mathcal G}_{\mathcal T}} }(\widehat{g}) \: = \\
& \int_{\widehat{{\mathcal G}_{\mathcal T}}} \widehat{\phi}_{s(\widehat{g})}^2 \: \widehat{\phi}_{r(\widehat{g})} \:
\left( (\pi^*\eta)_{r(\widehat{g})}, (\pi^*\xi)_{r(\widehat{g})} \right) \:
d\mu_{\widehat{{\mathcal G}_{\mathcal T}} }(\widehat{g}) \: = \notag \\
& \int_{F_{O(q)}\mathcal{T}} \int_{\widehat{{\mathcal G}_{\mathcal T}}^f} \widehat{\phi}_{s(\widehat{g})}^2 \:
\widehat{\phi}_f \: ((\pi^*\eta)_{f}, (\pi^*\xi)_{f}) \:  d\mu^f(\widehat{g}) \:
d\mu_{F_{O(q)}\mathcal{T}}(f) \:= \notag \\
& \int_{F_{O(q)}\mathcal{T}}    \widehat{\phi}_f \: \left( (\pi^*\eta)_{f}, (\pi^*\xi)_{f}) \right) \:
d\mu_{F_{O(q)}\mathcal{T}}(f) \: =
\int_{\mathcal{T}}    \phi_f \: (\eta_{p}, \xi_{p}) \:  d\mu_{\mathcal{T}}(p) \: = \: \langle \eta , \alpha \xi \rangle. \notag
\end{align}
This proves the proposition.
\end{proof}

\begin{corollary} \label{corollary6}
$P \: = \: \alpha \circ \beta$ is an  orthogonal projection on $ L^2({\mathcal T}; {\mathcal E}) $.
\end{corollary}

More explicitly,
\begin{equation} \label{3.16}
(P\eta)_p \: = \: \phi_p \: \int_{g \in \overline{{\mathcal G}_{\mathcal T}}^p}  \eta_{s(g)} \cdot g^{-1}
\: \phi_{s(g)} \: d\mu^p(g).
\end{equation}
This shows that $P$ comes from the action of the idempotent $g \rightarrow \phi_{s(g)} \: \phi_{r(g)}$ in the groupoid algebra
$C_c^\infty(\overline{{\mathcal G}_{\mathcal T}})$, which we also denote by $P$.

The maps $\alpha$ and $\beta$ establish an isomorphism between $\Image P$ and
$\left( L^2({\mathcal T}; {\mathcal E}) \right)^{\overline{{\mathcal G}_{\mathcal T}}} $.

\subsection{Spectral triples and the invariant Dirac-type operator} \label{subsection3.3}

Let $D_0$ be the operator on
$\left( L^2({\mathcal T}; {\mathcal E}) \right)^{\overline{{\mathcal G}_{\mathcal T}}}$
which is the restriction of the Dirac-type operator on ${\mathcal T}$ to
$\overline{{\mathcal G}_{\mathcal T}}$-invariant spinor fields.
Let $D_{\APS}$ denote the Dirac-type operator on
$L^2({\mathcal T}; {\mathcal E})$ with Atiyah-Patodi-Singer
(APS) boundary conditions
on $\partial \overline{{\mathcal T}}$
\cite{Atiyah-Patodi-Singer (1975)}.
It is a self-adjoint extension of $D$.
(We do not require a product geometry near $\partial \overline{{\mathcal T}}$.)
Note that $\Image(\alpha) \subset \Dom(D_{\APS})$, since
an element of $\Image(\alpha)$ has compact support in ${\mathcal T}$, i.e.
in the interior of $\overline{\mathcal T}$.

\begin{remark} \label{renark1}
In what follows, the choice of APS boundary conditions is not essential.
Any boundary condition which gives a self-adjoint operator would work
just as well. We invoke APS boundary conditions for clarity.
\end{remark}

\begin{proposition} \label{proposition3}
$(C^\infty_c(\overline{{\mathcal G}_{\mathcal T}}),
L^2({\mathcal T}; {\mathcal E}), D_{\APS})$
is a spectral triple of dimension $q$.
\end{proposition}
\begin{proof}
The action of $A \in C^\infty_c(\overline{{\mathcal G}_{\mathcal T}})$ on
$\eta \in L^2({\mathcal T}; {\mathcal E})$ is given by
\begin{equation} \label{action}
(A \eta)_p = \int_{g \in \overline{{\mathcal G}_{\mathcal T}}^p}
A(g) \: \eta_{s(g)} \cdot g^{-1} \: d\mu^p(g).
\end{equation}
As $A$ is compactly supported, there is a compact subset $K$ of ${\mathcal T}$
so that $\supp(A) \subset (s,r)^{-1}(K \times K)$. It follows that
the action of $C^\infty_c(\overline{{\mathcal G}_{\mathcal T}})$ on
$L^2({\mathcal T}; {\mathcal E})$ preserves $\Dom(D_{\APS})$.

Using (\ref{action}), it follows that $[D_{\APS}, A]$ is a bounded
operator on $L^2({\mathcal T}; {\mathcal E})$.
Thus $(C^\infty_c(\overline{{\mathcal G}_{\mathcal T}}),
L^2({\mathcal T}; {\mathcal E}), D_{\APS})$ is a spectral triple.
Finally, from \cite[Section 9]{Grubb (1999)}, the spectral triple
has dimension $q$ in the sense of
\cite[Chapter 4.2]{Connes (1994)}.
\end{proof}

\begin{proposition} \label{proposition4}
We have
\begin{equation} \label{3.17}
\beta \circ D_{\APS} \circ \alpha \: = \: D_0 \: - \: \frac{1}{2} \: c(\tau).
\end{equation}
\end{proposition}
\begin{proof}
Choose $\xi \in \left( L^2({\mathcal T}; {\mathcal E}) \right)^{\overline{{\mathcal G}_{\mathcal T}}}$. Then
\begin{equation} \label{3.18}
D_{\APS} (\alpha (\xi)) \: = \: D_{\APS}(\phi \xi) \: = \: c(d \phi) \: \xi \: + \: \phi \: D_{\APS}(\xi).
\end{equation}
Using the $\overline{{\mathcal G}_{\mathcal T}}$-invariance of the Dirac operator, we obtain
\begin{align} \label{3.19}
(\beta (D_{\APS} (\alpha (\xi))))_p \: = \: &
\int_{g \in \overline{{\mathcal G}_{\mathcal T}}^p}  (D_{\APS} (\alpha (\xi)))_{s(g)} \cdot g^{-1}
\: \phi_{s(g)} \: d\mu^p(g) \\
= \: &
\int_{g \in \overline{{\mathcal G}_{\mathcal T}}^p}  (c(d \phi) \: \xi)_{s(g)} \cdot g^{-1}
\: \phi_{s(g)} \: d\mu^p(g) \: + \notag \\
& \int_{g \in \overline{{\mathcal G}_{\mathcal T}}^p}  \phi_{s(g)} \: (D_{\APS} (\xi))_{s(g)} \cdot g^{-1}
\: \phi_{s(g)} \: d\mu^p(g) \notag \\
= \: &
c\left(\int_{\overline{{\mathcal G}_{\mathcal T}}^p}
(d \phi)_{s(g)} \cdot g^{-1}  \: \phi_{s(g)} \: d\mu^p(g) \right)\xi_p \: + \notag \\
& \left(\int_{\overline{{\mathcal G}_{\mathcal T}}^p}
\phi^2_{s(g)} \:   d\mu^p(g) \right) (D_{0}\xi)_p. \notag
\end{align}
Since
\begin{equation} \label{3.20}
\int_{\overline{{\mathcal G}_{\mathcal T}}^p} \phi^2_{s(g)} \:   d\mu^p(g) \: = \: 1,
\end{equation}
differentiation gives
\begin{align} \label{3.21}
0 \: & = \: 2 \int_{\overline{{\mathcal G}_{\mathcal T}}^p}
(d \phi)_{s(g)} \cdot g^{-1}  \: \phi_{s(g)} \: d\mu^p(g) \: + \:
\int_{\overline{{\mathcal G}_{\mathcal T}}^p}
\phi^2_{s(g)} \: \tau_p \: d\mu^p(g) \\
& = \:
2 \int_{\overline{{\mathcal G}_{\mathcal T}}^p}
(d \phi)_{s(g)} \cdot g^{-1}  \: \phi_{s(g)} \: d\mu^p(g) \: + \: \tau_p. \notag
\end{align}
The proposition follows.
\end{proof}

We define the invariant Dirac operator $D_{\inv}$ on
$\left( L^2({\mathcal T}; \mathcal{E}) \right)^{\overline{{\mathcal G}_{\mathcal T}}}$ by
\begin{equation} \label{3.22}
D_{\inv} \: = \: D_0 \: - \: \frac{1}{2}\: c(\tau).
\end{equation}

\begin{corollary} \label{corollary7}
$D_{\inv}$ is a self-adjoint Fredholm operator.
For all $\theta > 0$, the operator
$e^{- \: \theta D_{\inv}^2}$ is trace-class.
\end{corollary}
\begin{proof}
The operator $D_{\inv}$ is unitarily equivalent to $P \circ D_{\APS} \circ P$. As $P$ is an idempotent
in the groupoid algebra $C^\infty_c(\overline{{\mathcal G}_{\mathcal T}})$, it follows from
Proposition \ref{proposition3} that $D_{\inv}$ is self-adjoint and Fredholm.

It follows from  \cite[Theorem C]{Getzler-Szenes (1989)} that
$e^{- \: \theta \left[ (P D_{\APS} P)^2 + ((1-P) D_{\APS} (1-P))^2 \right]}$ is trace-class for
all $\theta > 0$.
Then $e^{- \: \theta (P D_{\APS} P)^2}$ is also trace-class.
\end{proof}

\begin{corollary} \label{corollary8}
If $\dim({\mathcal T})$ is even and $D$ is the Gauss-Bonnet operator $d \: + \: d^*$ then
$\Ind(D_{\inv})$ equals the basic Euler characteristic $\chi(M, {\mathcal F}; {\mathcal D}_M^{\frac12})$.
If $\dim({\mathcal T})$ is even and $D$ is the signature operator $d \: + \: d^*$ then
$\Ind(D_{\inv})$ equals the basic signature $\sigma(M, {\mathcal F}; {\mathcal D}_M^{\frac12})$.
\end{corollary}
\begin{proof}
As $e^{- \: \theta D_{\inv}^2}$ is trace-class,
we can apply standard Hodge theory.
\end{proof}

\begin{remark} \label{remark2}
Let $\ch_{JLO}(D_{\APS})$ be the
JLO cocycle \cite{Jaffe-Lesniewski-Osterwalder (1988)} for the spectral triple
$(C^\infty_c(\overline{{\mathcal G}_{\mathcal T}}), L^2({\mathcal T}; {\mathcal E}), D_{\APS})$
from Proposition \ref{proposition3}.
Then for any $t > 0$,
\begin{equation} \label{3.23}
\Index(D_{\inv}) \: = \: \langle \ch_{JLO}(t D_{\APS}), \ch(P) \rangle.
\end{equation}
One may hope to prove a transverse index theorem by computing
$\lim_{t \rightarrow 0} \langle \ch_{JLO}(t D_{\APS}), \ch(P) \rangle$
as a local expression.
As will become clear in the next section, there are problems with
this approach.
\end{remark}

Given a positive function $h \in (C^\infty({\mathcal T}))^{\overline{{\mathcal G}_{\mathcal T}}}$, we
can write $h \: = \: \sigma^* h_W$ for some  $h_W \in C(W)$.
The operator $D_0 \: - \: \frac12 \: c(\tau)$ on $L^2({\mathcal S}, d\mu_W)$ is
unitarily equivalent to the operator
$D_0 \: - \: \frac12 \: c(\tau \: - \: d \log h)$ on $L^2({\mathcal S}, h_W d\mu_W)$.

\begin{corollary} \label{corollary9}
If $[\tau] = 0$ in $\HH^1_{\inv}({\mathcal T})$ then up to a multiplicative
constant, there is a
unique positive $h \in (C^\infty({\mathcal T}))^{\overline{{\mathcal G}_{\mathcal T}}}$ so that $\tau \: = \:
d \log h$.
Hence in this case, the invariant Dirac operator $D_{\inv}$ is unitarily equivalent to
$D_0$ on $L^2({\mathcal S}, h_W d\mu_W)$.
\end{corollary}

\begin{example} \label{example8}
Continuing with Example \ref{example6}, suppose that $Z$ is equipped with a $G$-equivariant Clifford module ${\mathcal E}$.
By (\ref{2.8}), $\widehat{\tau} \: = \: d \log \widehat{\sigma}^* \widehat{{\mathcal V}}$,
where $\widehat{{\mathcal V}} \in C^\infty(\widehat{W})$ is
the function for which $\widehat{{\mathcal V}}(\widehat{w}) \: = \: \vol(\widehat{\sigma}^{-1}(\widehat{w}))$.
Then $\tau \: = \: d \log \sigma^* {\mathcal V}$, where ${\mathcal V} \in C({W})$ is defined by
$\widehat{{\mathcal V}} \: = \: \iota^* {\mathcal V}$. In particular,
$[\tau] \: = \: 0$
and $D_{\inv}$ is unitarily equivalent to
$D_0$ on $L^2({\mathcal S}, {\mathcal V} d\mu_W)$. Now
\begin{equation} \label{3.24}
{\mathcal V} d\mu_W \: = \: \iota_* \left( \widehat{{\mathcal V}} d\mu_{\widehat{W}} \right) \: = \:
\iota_* \widehat{\sigma}_* d\mu_{F_{O(q)}Z} \: = \:
{\sigma}_* \pi_* d\mu_{F_{O(q)}Z} \: = \:
\sigma_* d\mu_{Z}.
\end{equation}
Hence $D_{\inv}$ is unitarily equivalent to
$D_0$ on $L^2({\mathcal S}, \sigma_* d\mu_{Z})$, which is what one would expect.
\end{example}

\begin{remark} \label{remark3}
There are several approaches in the literature to the goal of
constructing a good self-adjoint basic Dirac-type operator.

Given a foliated manifold $(M, {\mathcal F})$ with a bundle-like metric
$g_M$ as in \cite[Remark 2.7(7)]{Moerdijk-Mrcun (2003)},
one can consider a normal Clifford module on $M$ and its holonomy-invariant
sections. With this approach, the natural inner product on the
holonomy-invariant sections involves the volume form of $g_M$.
In order to obtain a self-adjoint basic Dirac-type operator with this
approach, one must assume that the mean curvature form $\kappa$
of the foliated
manifold $(M, {\mathcal F})$ is a basic one-form
\cite{Glazebrook-Kamber (1991)}.  Note that the
mean curvature form $\kappa$, which lives on $M$, is distinct from the
mean curvature form $\tau$ in this paper, which lives on ${\mathcal T}$.

Still working on $M$, the problem of self-adjointness was
resolved by means of a modified basic Dirac-type operator,
involving the basic projection of $\kappa$
\cite{Habib-Richardson (2009)}. Given the
transverse metric, it was shown in \cite{Habib-Richardson (2009)}
that the spectrum
is independent of the particular choice of bundle-like metric.

In the present paper we work directly with the transverse
structure, so bundle-like metrics do not enter.  Presumably
our operator $D_{\inv}$ is unitarily equivalent to the
operator considered in \cite{Habib-Richardson (2009)}.

A different approach is to consider the operator $D_+$ mapping from
the positive-chirality holonomy-invariant sections to the
negative-chirality holonomy-invariant sections. One then obtains a
self-adjoint operator $D = D_+ + D_+^*$, albeit not an explicit one.
This is essentially the approach of \cite{El-Kacimi (1990)}.
Different choices of inner product will change the definition of
$D_+^*$ but will not affect $\Index(D_+)$.
\end{remark}

\section{The case of a compact group action} \label{section4}

In this section we analyze the index of a Dirac-type operator
when it acts on the $T^k$-invariant sections of a $T^k$-equivariant
Clifford module on a compact manifold $Z$.
In Subsection \ref{subsection4.1} we express the index
in terms of the Atiyah-Singer $G$-indices.
In Subsection \ref{subsection4.2} we discuss the problem in switching the
order of integration over $T^k$ and integration over the fixed-point set.
This turns out to be an issue about the nonuniformity of an
asymptotic expansion.

\subsection{An index formula} \label{subsection4.1}

Let
\begin{enumerate}
\item $\Gamma$ be a discrete group,
\item $G$ be a compact connected Lie group,
\item $i \: : \: \Gamma \rightarrow G$
be an injective homomorphism with dense image,
\item $d\mu_G$ be normalized Haar measure on $G$,
\item $Z$ be an
even-dimensional compact connected Riemannian manifold
on which $G$ acts isometrically,
\item ${\mathcal E}$ be a $G$-equivariant Clifford module on $Z$, and
\item $Y$ be a compact connected manifold with $\pi_1(Y, y_0) \: = \: \Gamma$.
\end{enumerate}

Put $M \: = \: (\widetilde{Y} \times Z)/\Gamma$,
where $\Gamma$ acts diagonally on $\widetilde{Y} \times Z$. 
Then $M$ has a Riemannian foliation with complete
transversal $Z$.  Now $\left( L^2(Z; {\mathcal E}) \right)^\Gamma \: = \:
\left( L^2(Z; {\mathcal E}) \right)^G$. Let $D$ be the Dirac-type operator on
$L^2(Z; {\mathcal E})$ and let $D_{\inv}$ be its restriction to
$\left( L^2(Z; {\mathcal E}) \right)^G$. Given $g \in G$, let $\Index(g) \in \R$ denote
its $G$-index, i.e.
$\Index(g) \: = \: \str g \big|_{\Ker(D)}$, where $\str$ denote the supertrace.

\begin{lemma} \label{lemma8}
$\Index(D_{\inv}) \: = \: \int_G \Index(g) \: d\mu_G(g)$.
\end{lemma}
\begin{proof}
The finite-dimensional $\Z_2$-graded vector space $\Ker(D)_\pm$ has an orthogonal decomposition
\begin{equation} \label{4.1}
\Ker(D)_\pm \: = \: \Ker(D)_\pm^G \oplus (\Ker(D)_\pm^G)^\perp.
\end{equation}
Then
\begin{align} \label{4.2}
\Index(D_{\inv}) \: & = \: \dim(\Ker(D)_+^G) - \dim(\Ker(D)_-^G) \\
& = \:
\int_G \tr(g) \big|_{\Ker(D)_+} \: d\mu_G(g) \: - \:
\int_G \tr(g) \big|_{\Ker(D)_-} \: d\mu_G(g) \notag \\
& = \:
\int_G \str(g) \big|_{\Ker(D)} \: d\mu_G(g)
\: = \: \int_G \Index(g) \: d\mu_G(g). \notag
\end{align}
This proves the lemma.
\end{proof}

Let $L(g) \in \R$ be the Atiyah-Segal-Singer Lefschetz-type formula for
$\Index(g)$ \cite{Atiyah-Singer (1968)},\cite[Chapter 6]{Berline-Getzler-Vergne (1992)}.
It is the integral of a certain characteristic form over the fixed-point set $Z^g$.
Then
\begin{equation} \label{4.3}
\Index(D_{\inv}) \: = \: \int_G L(g) \: d\mu_G(g).
\end{equation}

Let $T^k$ be a maximal torus for $G$. Since $L(g)$ is conjugation-invariant, the Weyl integral
formula gives
\begin{equation} \label{4.4}
\Index(D_{\inv}) \: = \: \frac{1}{|\Weyl|} \int_{T^k} L(g) \:
\det \left( \Ad(g^{-1}) - I \right) \Big|_{{\mathfrak g}/{\mathfrak t}^k} \: d\mu_{T^k}(g).
\end{equation}

\subsection{Nonuniformity in the localized short-time expansion} \label{subsection4.2}

We now specialize to the case $G = T^k$.

For simplicity, suppose that $Z$ has a $T^k$-invariant spin structure with
spinor bundle $S^Z$, and ${\mathcal E} = S^Z \otimes {\mathcal W}$ for some
$\Z_2$-graded
$G$-equivariant vector bundle ${\mathcal W}$. Suppose further that each
connected component of
$Z^g$ has a spin
structure. Let $S_N$ denote the normal spinor bundle. Put
$\ch_{\mathcal W}(g) \: = \: \trs \left( g e^{\frac{\sqrt{-1}}{2\pi} F^{\mathcal W}} \right)$.
From \cite[Chapter 6.4]{Berline-Getzler-Vergne (1992)},
\begin{equation} \label{4.5}
L(g) \: = \: \int_{Z^g} \widehat{A}(Z^g) \: \frac{\ch_{\mathcal W}(g)}{\ch_{S_N}(g)}.
\end{equation}
(In order to simplify notation, we have omitted some signs and
powers of $2 \pi i$ in the formula from
\cite[Chapter 6.4]{Berline-Getzler-Vergne (1992)}.)
From (\ref{4.3}), it is clear that the only submanifolds  of $Z$ that contribute
to the integral are the connected components $\{Z^{T^k}_i\}$ of
the fixed-point set $Z^{T^k}$, as the integrals over
the other submanifolds will be of measure zero in $G$.
Then
\begin{equation} \label{4.6}
\Index(D_{\inv}) \: = \: \int_{T^k} \sum_i \int_{Z^{T^k}_i}
\widehat{A}(Z^{T^k}_i) \: \frac{\ch_{\mathcal W}(g)}{\ch_{S_N}(g)}
 \: d\mu_{T^k}(g).
\end{equation}

\begin{example} \label{example9}
Suppose that $Z$ is an oriented manifold whose dimension is divisible by
four. Suppose that $Z$ has an $S^1$-action with
isolated fixed points $\{z_k\}$. Let the $S^1$-action on $T_{z_k}(Z)$
be decomposable as
\begin{equation} \label{4.7}
e^{i\theta} \rightarrow \bigoplus_{l=1}^{\dim(Z)/2}
\begin{pmatrix}
\cos(n_{k,l} \theta) & - \sin(n_{k,l} \theta) \\
\sin(n_{k,l} \theta) & \cos(n_{k,l} \theta).
\end{pmatrix}
\end{equation}
Let $D_{\inv}$ be the signature operator
acting on $S^1$-invariant forms.  Then
\begin{equation} \label{4.8}
\Index(D_{\inv}) \: = \: (-1)^{\frac{\dim(Z)}{4}} \:
\int_{S^1} \sum_k \prod_{l=1}^{\dim(Z)/2}
\cot(n_{k,l} \theta/2) \: \frac{d\theta}{2\pi};
\end{equation}
compare with
\cite[Theorem 6.27]{Atiyah-Bott (1968)}.

Note that in (\ref{4.8}), the sum over $k$ and the integral over $S^1$
generally cannot be interchanged.  For example, suppose that
$\dim(Z)=4$, $k=1$ and $n_{1,1} = n_{1,2} = 1$. Then the contribution
from the fixed point $z_1$ is
\begin{equation} \label{4.9}
- \: \int_{S^1} \cot^2(\theta/2) \: \frac{d\theta}{2\pi} \: = \: - \: \infty.
\end{equation}
What happens is that there are cancellations among the various fixed points.
This cancellation
is ensured by the fact that $L(g)$ is uniformly bounded in $g \in S^1$.
So the integral
(\ref{4.8}) makes sense but one cannot switch the order of
integration and summation. This is a problem if one wants a local
formula for $\Index(D_{\inv})$.

To elaborate on this phenomenon, for any $t > 0$ we can use Lemma \ref{lemma8}
to write
\begin{equation} \label{4.10}
\Index(D_{\inv}) \: = \: \int_{S^1} \Str \left( g \cdot e^{- t D^2} \right)
d\mu_{S^1}(g)
\: = \: \int_{S^1} \int_Z \str e^{- t D^2}(z,zg) \: d\mu_Z(z)\: d\mu_{S^1}(g).
\end{equation}
If $\phi_i$ is an $S^1$-invariant bump function with support near
the fixed point $z_i$ then
\begin{equation} \label{4.11}
\Index(D_{\inv}) \: = \: \sum_i \lim_{t \rightarrow 0}
\int_{S^1} \int_Z \str e^{- t D^2}(z,zg) \: \phi_i(z) \: d\mu_Z(z)
\: d\mu_{S^1}(g).
\end{equation}
By general arguments \cite{Bruening-Heintze (1984)}, there is an asymptotic
expansion
\begin{equation} \label{4.12}
\int_{S^1} \int_Z \str e^{- t D^2}(z,zg) \: \phi_i(z) \: d\mu_Z(z)
\: d\mu_{S^1}(g) \: \sim \: t^{- \dim(Z)/2}
\sum_{j,k =0}^\infty a_{i,j,k} t^{j/2} \: (\log t)^k
\end{equation}
and so
\begin{equation} \label{4.13}
\Index(D_{\inv}) \: = \: \sum_i a_{i,\dim(Z)/2,0}.
\end{equation}

On the other hand, for a fixed $g \in S^1$ there is a computable limit
\begin{equation} \label{4.14}
\lim_{t \rightarrow 0} \int_Z \str e^{- t D^2}(z,zg) \: \phi_i(z) \:
d\mu_Z(z),
\end{equation}
which becomes an integral over $Z^g$.
If one could commute the $\lim_{t \rightarrow 0}$ with
the integration over $g \in S^1$ on
\begin{equation} \label{4.15}
\int_Z \str e^{- t D^2}(z,zg) \: \phi_i(z) \: d\mu_Z(z)
\end{equation}
 then one would conclude that
the asymptotic expansion in (\ref{4.12}) starts at the $t^0$-term, and
that the coefficient of the $t^0$-term is
\begin{equation} \label{4.16}
\int_{S^1}
\lim_{t \rightarrow 0} \int_Z \str e^{- t D^2}(z,zg) \: \phi_i(z) \:
d\mu_Z(z) \: d\mu_{S^1}(g).
\end{equation}
One finds in examples that neither of these are true.
Related phenomena for local traces (as opposed to supertraces) of
basic heat kernels were noted in
\cite{Richardson (1998)}

The underlying reason for the lack of uniformity, in the expansions
with respect to $t$ and $g$, is that the fixed-point set $Z^g$ can
vary wildly in $g$. For example, if the $S^1$-action is effective
then $Z^e = Z$, while $Z^g$ has codimension at least one for any
$g \neq e$, no matter how close $g$ may be to $e$.
\end{example}

\section{The case of abelian Molino sheaf : a
delocalized index theorem} \label{section5}

In this section we prove a delocalized index theorem for $D_{\inv}$
under the assumption that the Molino sheaf is a
holonomy-free sheaf of abelian Lie algebras, and an additional
connectedness assumption on the isotropy groups.
The index formula will be localized in Section \ref{section6}.

In Subsection \ref{subsection5.1} we use local models for the transverse structure
of a Riemannian foliation to write a formula for $\Index(D_{\inv})$
in terms of a parametrix.  As indicated in the preceding section,
there are problems in directly computing the $t \rightarrow 0$ limit of this
index formula, as a local expression.
Hence we use a delocalized approach.
In Subsection \ref{subsection5.2} we rewrite the index formula in terms of the
averaging of a certain almost-periodic function $F_{t, \epsilon}$
that is defined on
the abelian Lie algebra.  The number $F_{t, \epsilon}(X)$ is defined
by a Kirillov-type formula. We show that it is independent of $t$ and
$\epsilon$. In Subsection \ref{subsection5.3} we compute the $t \rightarrow 0$
limit of $F_{t,\epsilon}$.

\subsection{Parametrix} \label{subsection5.1}

Hereafter
we assume that Lie algebra ${\mathfrak g}$ of the Molino sheaf is the abelian
Lie algebra
$\R^k$.
We also assume that the Lie algebroid
$\overline{{\mathfrak g}_{\mathcal T}}$ is a trivial flat
$\R^k$-bundle,
i.e. has trivial holonomy.

Recall the sheaf ${\mathcal S}_2$ on $W$ from Subsection \ref{subsection3.2}.
The invariant operator $D_{\inv}$ is a self-adjoint operator on the
global sections ${\mathcal S}_2(W)$. We will compute the index of $D_{\inv}$ by
constructing a parametrix for $D_{\inv}$. The parametrix will be formed using
a suitable open cover of $W$, along with a partition of unity.

Corollary \ref{corollary9} gives a measure  $h_W \: d\mu_W$ which is canonical
up to a multiplicative constant.

Given $p \in {\mathcal T}$, let $K$ be the isotropy group of
$\overline{{\mathcal G}_{\mathcal T}}$ at $p$.
We assume that $K$ is connected, so $K = T^l$ for some $0 \le l \le k$.
From Subsection \ref{subsection2.6}, there is an invariant neighborhood
$U$ of the orbit ${\mathcal O}_p$
so that the restriction of $\overline{{\mathcal G}_{\mathcal T}}$ to
$U$ is weakly equivalent, as an \'etale groupoid, to the cross-product groupoid
$(B(V) \times_K G) \rtimes G_\delta$. Here $G$ is a
$k$-dimensional connected abelian
Lie group containing $K$, $V$ is a
representation space of $K$ and $B(V)$ is a metric ball in $V$.
The manifold $B(V) \times_K G$ acquires a $G$-invariant Riemannian
metric from the Riemannian foliation.

If $l < k$ then we can
quotient out
by a lattice in $G/K$, so in any case we can assume that $G = T^k$.
Note that there is some freedom in exactly which lattice is chosen.

There is an embedding $B(V)/K \rightarrow W$ and a quotient map
$\sigma \: : \: (B(V) \times_K G) \rightarrow W$.
From Example \ref{example8}, $\sigma_* d\mu_{B(V) \times_K G}$ is a constant
times $(h_W d\mu_W) \big|_{B(V)/K}$.
We will want to fix a normalization for the measure $h_W d\mu_W$.
The normalization that we use
will depend on whether or not there are any points in ${\mathcal T}$
with maximal isotropy group.

Recall from Example \ref{example8} that
in the local model, the relevant measure is ${\mathcal V} \: d\mu_W$. Here
${\mathcal V}$ satisfies $\widehat{{\mathcal V}} = \iota^* {\mathcal V}$, where
$\widehat{{\mathcal V}} \in C^\infty(\widehat{W})$ is
the function for which $\widehat{{\mathcal V}}(\widehat{w}) \: = \:
\vol(\widehat{\sigma}^{-1}(\widehat{w}))$.
If the isotropy group at a point $p \in {\mathcal T}$
is $T^k$ then $\widehat{\sigma}^{-1}(\widehat{w})$ is a
(free) $T^k$-orbit in the frame $F_{O(q)}{\mathcal T}_p$. As its
volume is canonical, i.e. independent of the choice of local model,
we can consistently normalize $h_W d\mu_W$ in a local model
with $K = T^l$ to be ${\mathcal V} \: d\mu_W$.

Using the connectedness of $W$,
this determines $h_W \: d\mu_W$ globally.
Having now normalized $h_W \: d\mu_W$, there may be
local models with $l < k$. For these local models,
we use the freedom in the choice of lattice in $G/K$
to ensure that
$\sigma_* d\mu_{B(V) \times_K G} \: = \:
(h_W d\mu_W) \big|_{B(V)/K}$.

If there are no points in ${\mathcal T}$ with isotropy $T^k$ then
we normalize $h_W \: d\mu_W$ by requiring that
$\int_W h_W \: d\mu_W \: = \: 1$. We can then use the
freedom in the choice of the lattice in $G/K$ to ensure that in
each local model,
$\sigma_* d\mu_{B(V) \times_K G} \: = \:
(h_W d\mu_W) \big|_{B(V)/K}$.

We can find \\
1. Finite open coverings $\{U_\alpha\}$ and $\{U_\alpha^\prime\}$ of $W$, where
$U_\alpha$ has compact closure in $U_\alpha^\prime$, so that
the restriction of $\overline{{\mathcal G}_{\mathcal T}}$, to the preimage of
$U_\alpha^\prime $ in ${\mathcal T}$, is equivalent to
$(B(V_\alpha) \times_{K_\alpha}
G_\alpha) \rtimes G_{\alpha,\delta}$.
(Here $G_\alpha$ is isomorphic to $T^k$.) \\
2. A subordinate partition of unity $\{\eta_\alpha\}$ to $\{U_\alpha\}$ so that
each $\iota^* \eta_\alpha$ is smooth on $\widehat{W}$, \\
3. Functions $\{\rho_\alpha\}$ with support in $U_\alpha^\prime$ so that
each $\iota^* \rho_\alpha$ is smooth on $\widehat{W}$, and
$\rho_\alpha \eta_\alpha = \eta_\alpha$, i.e. $\rho_\alpha \big|_{\supp(\eta_\alpha)} = 1$.

For each $\alpha$, we choose a closed Riemannian manifold $Y_\alpha$
with an isometric $G_\alpha$-action
so that there
is an isometric $G_\alpha$-equivariant embedding
$B(V_\alpha) \times_{K_\alpha} G_\alpha \subset Y_\alpha$.
This can be done, for example, by taking a slight extension of
$B(V_\alpha)$ to a larger ball $B^\prime_\alpha \subset V_\alpha$,
taking the double
of $\overline{B^\prime_\alpha} \times_{K_\alpha} G_\alpha$ and smoothing the
metric. (Alternatively, we could work directly with APS boundary
conditions on
$\overline{B(V_\alpha)} \times_{K_\alpha} G_\alpha$, at the price of
having to deal with manifolds-with-boundary.) We can also assume that
the restriction of ${\mathcal E}$
to $B(V_\alpha) \times_{K_\alpha} G_\alpha$
extends to ${\mathcal E}_\alpha$ on $Y_\alpha$.

Let $D_\alpha$ denote the
Dirac-type operator on $Y_\alpha$.
Let $D_{\inv,\alpha}$ be the
restriction of $D_\alpha$ to
$\left( L^2(Y_\alpha, {\mathcal E}_{Y_\alpha}) \right)^{G_\alpha}$.

Given $t > 0$, put
\begin{equation} \label{5.1}
Q_\alpha = \frac{1-e^{-tD_\alpha^2}}{D_\alpha^2} D_\alpha =
\int_0^t e^{- s D_\alpha^2} \: D_\alpha \: ds
\end{equation}
and
\begin{equation} \label{5.2}
Q_{\inv,\alpha} = \frac{1-e^{-tD_{\inv,\alpha}^2}}{D_{\inv,\alpha}^2} D_{\inv,\alpha} =
\int_0^t e^{-s D_{\inv,\alpha}^2} \: D_{\inv,\alpha} \: ds.
\end{equation}

We let $\widetilde{\eta}_\alpha$ be the extension by zero of
$\sigma^* \eta_\alpha$ to $Y_\alpha$, and similarly for
$\widetilde{\rho}_\alpha$.

\begin{proposition} \label{proposition5}
$\sum_\alpha \rho_\alpha Q_{\inv,\alpha}^\mp \eta_\alpha$ is a parametrix for $D_{\inv}^\pm$.
Also, for all $t > 0$, formally
\begin{equation} \label{5.3}
\Ind(D_{\inv}) \: = \:
\sum_\alpha \Tr_s \left( e^{-t D_{\inv,\alpha}^2} \eta_\alpha \right) + \frac12 \sum_\alpha
\Tr_s \left( Q_{\inv,\alpha} [D_{\inv,\alpha}, \eta_\alpha] \right),
\end{equation}
or more precisely,
\begin{equation} \label{5.4}
\Ind(D_{\inv}) \: = \: \sum_\alpha \Tr_s \left( e^{-t D_{\inv,\alpha}^2} \eta_\alpha \right) + \frac12 \sum_{\alpha,\beta}
\Tr_s \left( \rho_\alpha (Q_{\inv,\alpha} - Q_{\inv,\beta} ) \eta_\beta\: [D_{\inv,\alpha}, \eta_\alpha] \right).
\end{equation}
\end{proposition}
\begin{proof}
First, we have
\begin{align} \label{5.5}
D^-_\alpha \widetilde{\rho}_\alpha Q^+_\alpha \widetilde{\eta}_\alpha \: & = \:
[D^-_\alpha, \widetilde{\rho}_\alpha] Q^+_\alpha \widetilde{\eta}_\alpha \: + \:
\widetilde{\rho}_\alpha D^-_\alpha Q^+_\alpha \widetilde{\eta}_\alpha \\
& = \:
[D^-_\alpha, \widetilde{\rho}_\alpha] Q^+_\alpha \widetilde{\eta}_\alpha \: + \:
\widetilde{\rho}_\alpha
\left( 1 - e^{-tD^-_\alpha D^+_\alpha} \right) \widetilde{\eta}_\alpha \notag \\
& = \:
\widetilde{\eta}_\alpha +
[D^-_\alpha, \widetilde{\rho}_\alpha] Q^+_\alpha \widetilde{\eta}_\alpha \: - \:
\widetilde{\rho}_\alpha
e^{-tD^-_\alpha D^+_\alpha} \widetilde{\eta}_\alpha. \notag
\end{align}
The Schwartz kernel of
$[D^-_\alpha, \widetilde{\rho}_\alpha] Q^+_\alpha \widetilde{\eta}_\alpha$ is
\begin{equation} \label{5.6}
c^-(d\widetilde{\rho}_\alpha(p)) \: Q^+_\alpha(p,p^\prime) \:
\widetilde{\eta}_\alpha(p^\prime).
\end{equation}
As $Q^+_\alpha$ is a pseudodifferential operator, and
$d\widetilde{\rho}_\alpha(p) \: \widetilde{\eta}_\alpha(p^\prime)$
vanishes in a neighborhood of the diagonal $p = p^\prime$, it follows that
$\widetilde{\eta}_\alpha \: - \:
D^-_\alpha (\widetilde{\rho}_\alpha Q^+_\alpha \widetilde{\eta}_\alpha)$ is
a smoothing operator on
$L^2(Y_\alpha; {\mathcal E}_{Y_\alpha}^+)$. In particular,
$\widetilde{\eta}_\alpha \: - \:
D^-_\alpha (\widetilde{\rho}_\alpha Q^+_\alpha \widetilde{\eta}_\alpha)$
is trace-class on $L^2(Y_\alpha; {\mathcal E}_{Y_\alpha}^+)$ and so
its restriction to $\left( L^2(Y_\alpha; {\mathcal E}_{Y_\alpha}^+)
\right)^{G_\alpha}$ is also trace-class.  Hence the operator
\begin{align} \label{5.7}
I \: - \: D^-_{\inv} \sum_\alpha
{\rho}_\alpha Q^+_{\inv,\alpha} {\eta}_\alpha \: & = \:
\sum_\alpha \left(
{\eta}_\alpha \: - \:
D^-_{\inv,\alpha} {\rho}_\alpha Q^+_{\inv,\alpha} {\eta}_\alpha \right) \\
& = \:
\sum_\alpha \left( {\rho}_\alpha
e^{-tD^-_{\inv,\alpha} D^+_{\inv,\alpha}} {\eta}_\alpha \: - \:
[D^-_{\inv,\alpha}, {\rho}_\alpha] Q^+_{\inv,\alpha} {\eta}_\alpha
\right) \notag
\end{align}
is also trace-class.  This shows that $\sum_\alpha
{\rho}_\alpha Q^+_{\inv,\alpha} {\eta}_\alpha$ is a right parametrix for
$D^-_{\inv}$. Hence it is also a left parametrix.

Similarly,
\begin{align} \label{5.8}
\widetilde{\rho}_\alpha Q^+_\alpha \widetilde{\eta}_\alpha D^-_\alpha \: & = \:
\widetilde{\rho}_\alpha Q^+_\alpha D^-_\alpha \widetilde{\eta}_\alpha \: - \:
\widetilde{\rho}_\alpha Q^+_\alpha [D^-_\alpha, \widetilde{\eta}_\alpha] \\
& = \:
\widetilde{\rho}_\alpha
\left( 1 - e^{-tD^+_\alpha D^-_\alpha} \right) \widetilde{\eta}_\alpha \: - \:
\widetilde{\rho}_\alpha Q^+_\alpha [D^-_\alpha, \widetilde{\eta}_\alpha]
 \notag \\
& = \: \widetilde{\eta}_\alpha \: - \: \widetilde{\rho}_\alpha
e^{-tD^+_\alpha D^-_\alpha} \widetilde{\eta}_\alpha \: - \:
\widetilde{\rho}_\alpha Q^+_\alpha [D^-_\alpha, \widetilde{\eta}_\alpha].
 \notag
\end{align}
Then
\begin{equation} \label{5.9}
I \: - \: \left( \sum_\alpha
{\rho}_\alpha Q^+_{\inv,\alpha} {\eta}_\alpha \right) D^-_{\inv}
\: = \:
\sum_\alpha \left( {\rho}_\alpha
e^{-tD^+_{\inv,\alpha} D^-_{\inv,\alpha}} {\eta}_\alpha \: + \:
{\rho}_\alpha Q^+_{\inv,\alpha} [D^-_{\inv,\alpha}, {\eta}_\alpha]
\right).
\end{equation}

Changing signs in (\ref{5.7}) and (\ref{5.9}) gives
\begin{equation} \label{5.10}
I \: - \: D^+_{\inv} \sum_\alpha
{\rho}_\alpha Q^-_{\inv,\alpha} {\eta}_\alpha \: = \:
\sum_\alpha \left( {\rho}_\alpha
e^{-tD^+_{\inv,\alpha} D^-_{\inv,\alpha}} {\eta}_\alpha \: - \:
[D^+_{\inv,\alpha}, {\rho}_\alpha] Q^-_{\inv,\alpha} {\eta}_\alpha
\right)
\end{equation}
and
\begin{equation} \label{5.11}
I \: - \: \left( \sum_\alpha
{\rho}_\alpha Q^-_{\inv,\alpha} {\eta}_\alpha \right) D^+_{\inv}
\: = \:
\sum_\alpha \left( {\rho}_\alpha
e^{-tD^-_{\inv,\alpha} D^+_{\inv,\alpha}} {\eta}_\alpha \: + \:
{\rho}_\alpha Q^-_{\inv,\alpha} [D^+_{\inv,\alpha}, {\eta}_\alpha]
\right).
\end{equation}

Now
\begin{equation} \label{5.12}
\Index(D_{\inv}) \: = \: \Tr \left(
I \: - \: \left( \sum_\alpha
{\rho}_\alpha Q^-_{\inv,\alpha} {\eta}_\alpha \right) D^+_{\inv} \right)
\: - \:
\Tr \left(
I \: - \: D^+_{\inv} \sum_\alpha
{\rho}_\alpha Q^-_{\inv,\alpha} {\eta}_\alpha
\right)
\end{equation}
and
\begin{equation} \label{5.13}
- \: \Index(D_{\inv}) \: = \: \Tr \left(
I \: - \: \left( \sum_\alpha
{\rho}_\alpha Q^+_{\inv,\alpha} {\eta}_\alpha \right) D^-_{\inv} \right)
\: - \:
\Tr \left(
I \: - \: D^-_{\inv} \sum_\alpha
{\rho}_\alpha Q^+_{\inv,\alpha} {\eta}_\alpha
\right)
\end{equation}
Hence
\begin{align} \label{5.14}
\Index(D_{\inv}) \: = \: & \frac12 \Str \left(
I \: - \: \left( \sum_\alpha
{\rho}_\alpha Q_{\inv,\alpha} {\eta}_\alpha \right) D_{\inv} \right)
\: + \\
& \frac12 \Str \left(
I \: - \: D_{\inv} \sum_\alpha
{\rho}_\alpha Q_{\inv,\alpha} {\eta}_\alpha
\right). \notag
\end{align}
Equations (\ref{5.7})-(\ref{5.11}) now give
\begin{align} \label{5.15}
\Index(D_{\inv}) \: = \: & \sum_\alpha \Str
\left( {\rho}_\alpha
e^{-tD^2_{\inv,\alpha}} {\eta}_\alpha \right) \: + \:
\frac12 \: \sum_\alpha \Str \left(
{\rho}_\alpha Q_{\inv,\alpha} [D_{\inv,\alpha}, {\eta}_\alpha]
\right) \: - \: \\
& \frac12 \: \sum_\alpha \Str \left(
 [D_{\inv,\alpha}, {\rho}_\alpha] Q_{\inv,\alpha} {\eta}_\alpha
\right). \notag
\end{align}

By formal manipulations,
\begin{align} \label{5.16}
\Index(D_{\inv}) \: = \: & \sum_\alpha \Str
\left(
e^{-tD^2_{\inv,\alpha}} {\eta}_\alpha
{\rho}_\alpha  \right) \: + \:
\frac12 \: \sum_\alpha \Str \left(
Q_{\inv,\alpha} [D_{\inv,\alpha}, {\eta}_\alpha]
{\rho}_\alpha \right) \: + \: \\
& \frac12 \: \sum_\alpha \Str \left(
Q_{\inv,\alpha} {\eta}_\alpha  [D_{\inv,\alpha}, {\rho}_\alpha]
\right) \notag \\
\: = \: & \sum_\alpha \Str
\left( e^{-tD^2_{\inv,\alpha}} {\eta}_\alpha
{\rho}_\alpha  \right) \: + \:
\frac12 \: \sum_\alpha \Str \left(
Q_{\inv,\alpha} [D_{\inv,\alpha}, {\eta}_\alpha {\rho}_\alpha]
\right) \notag \\
\: = \: & \sum_\alpha \Str
\left(
e^{-tD^2_{\inv,\alpha}} {\eta}_\alpha \right) \: + \:
\frac12 \: \sum_\alpha \Str \left(
Q_{\inv,\alpha} [D_{\inv,\alpha}, {\eta}_\alpha]
\right). \notag
\end{align}
The last term in (\ref{5.16}) actually makes sense because
$\sum_\alpha d\eta_\alpha \: = \: 0$, so the
computation of
\begin{equation} \label{5.17}
\sum_\alpha \Str \left( Q_{\inv,\alpha} [D_{\inv,\alpha}, {\eta}_\alpha]
\right) \: = \:
\sum_\alpha \Str \left( Q_{\inv,\alpha} c(d{\eta}_\alpha)
\right)
\end{equation}
happens away from the diagonal.
To see this more clearly,  we can write
\begin{align} \label{5.18}
\sum_\alpha \Str \left( Q_{\inv,\alpha} [D_{\inv,\alpha}, {\eta}_\alpha]
\right) \: & = \:
\sum_{\alpha,\beta}
\Tr_s \left( Q_{\inv,\alpha} \eta_\beta\: [D_{\inv,\alpha}, \eta_\alpha] \right) \\
& = \:
\sum_{\alpha,\beta}
\Tr_s \left( (Q_{\inv,\alpha} - Q_{\inv,\beta} ) \eta_\beta\: [D_{\inv,\alpha}, \eta_\alpha] \right) \notag \\
& = \:
\sum_{\alpha,\beta}
\Tr_s \left( (Q_{\inv,\alpha} - Q_{\inv,\beta} ) \eta_\beta\: [D_{\inv,\alpha}, \eta_\alpha] \rho_\alpha \right) \notag \\
& = \:
\sum_{\alpha,\beta}
\Tr_s \left( \rho_\alpha (Q_{\inv,\alpha} - Q_{\inv,\beta} ) \eta_\beta\: [D_{\inv,\alpha}, \eta_\alpha] \right). \notag
\end{align}
The latter expression is clearly well-defined.

This proves the proposition.
\end{proof}

In what follows we will use the equation (\ref{5.3}) when, to justify things
more formally, one could use (\ref{5.4}) instead.

\subsection{Averaging over the Lie algebra} \label{subsection5.2}

Fix a Haar measure $d\mu_{\mathfrak g}$ on
${\mathfrak g} = \R^k$.
If $F \in C^\infty(\R^k)$ is a finite sum of periodic functions,
put
\begin{equation} \label{5.19}
AV_X F(X) = \lim_{R \rightarrow \infty} \frac{\int_{B(0,R)} F(X) \: d\mu_{\mathfrak g}(X)}{\int_{B(0,R)} 1 \: d\mu_{\mathfrak g}(X)}.
\end{equation}
Equivalently, if $\{L_j\}$ is a finite collection of lattices in
$\R^k$
and
\begin{equation} \label{5.20}
F(X) \: = \: \sum_j \sum_{v \in L_j} c_{j,v} \: e^{2\pi \sqrt{-1} v \cdot X}
\end{equation}
is a representation of $F$ as a finite sum of periodic functions then
$AV_X F(X) \: = \: \sum_j c_{j,0}$, the sum of the coefficients of $1$.

Given
$X \in \R^k$, we also let $X$ denote the corresponding vector field
on $Y_\alpha$.
Let $X^*$ denote the dual $1$-form and
let ${\mathcal L}_X$ denote Lie differentiation
with respect to $X$.
The moment $\mu(X)$ of $X \in \R^k$ is defined by
$\mu(X) = {\mathcal L}_X - \nabla_X$. It is a skew-adjoint endomorphism
of $TY_\alpha$.

\begin{proposition} \label{proposition6}
We have
\begin{equation} \label{5.21}
\sum_\alpha \Tr_s \left( e^{-t D_{\inv,\alpha}^2} \eta_\alpha \right)
 =  AV_X \sum_\alpha \Tr_s
\left( e^{-(t D_{\alpha}^2 + {\mathcal L}_X)}
\widetilde{\eta}_\alpha \right)
\end{equation}
and
\begin{equation} \label{5.22}
\sum_\alpha
\Tr_s \left( Q_{\inv,\alpha}
[D_{\inv,\alpha}, \eta_\alpha] \right) =
AV_X \sum_\alpha \int_0^t \Tr_s
\left( e^{-(sD_{\alpha}^2 + {\mathcal L}_X)} D_{\alpha}
[D_{\alpha}, \widetilde{\eta}_\alpha] \right) \: ds.
\end{equation}
\end{proposition}
\begin{proof}
First,
\begin{equation} \label{5.23}
\int_{Y_\alpha}  \trs
\left( e^{-t D_{\alpha}^2} \widetilde{\eta}_\alpha \right)(p, p e^{-X}) \: d\mu_{Y_\alpha}(p)
\end{equation}
is a periodic function in $X$. From (\ref{5.19}),
\begin{align} \label{5.24}
\sum_\alpha \Tr_s \left( e^{-t D_{\inv,\alpha}^2} \eta_\alpha \right) & = AV_X \sum_\alpha
\int_{Y_\alpha}  \trs
\left( e^{-t D_{\alpha}^2} \widetilde{\eta}_\alpha \right)(p, p e^{-X}) \: d\mu_{Y_\alpha}(p) \\
& =  AV_X \sum_\alpha
\int_{Y_\alpha} \trs
\left( e^{-(t D_{\alpha}^2 + {\mathcal L}_X)} \widetilde{\eta}_\alpha \right)(p, p) \: d\mu_{Y_\alpha}(p) \notag \\
& =  AV_X \sum_\alpha \Tr_s
\left( e^{-(t D_{\alpha}^2 + {\mathcal L}_X)} \widetilde{\eta}_\alpha \right). \notag
\end{align}
Similarly,
\begin{align} \label{5.25}
& \sum_\alpha
\Tr_s \left( Q_{\inv,\alpha}
[D_{\inv,\alpha}, \eta_\alpha] \right) = \\
& AV_X \sum_\alpha \int_0^t
\int_{Y_\alpha} \trs
\left( e^{-sD_{\alpha}^2} D_{\alpha}
[D_{\alpha}, \widetilde{\eta}_\alpha] \right)(p, p e^{-X}) \: d\mu_{Y_\alpha}(p) \: ds = \notag \\
& AV_X \sum_\alpha \int_0^t
\int_{Y_\alpha} \trs
\left( e^{-(sD_{\alpha}^2 + {\mathcal L}_X)} D_{\alpha}
[D_{\alpha}, \widetilde{\eta}_\alpha] \right)(p, p) \: d\mu_{Y_\alpha}(p) \: ds = \notag \\
& AV_X \sum_\alpha \int_0^t \Tr_s
\left( e^{-(sD_{\alpha}^2 + {\mathcal L}_X)} D_{\alpha}
[D_{\alpha}, \widetilde{\eta}_\alpha] \right) \: ds. \notag
\end{align}
This proves the proposition.
\end{proof}

Note that ${\mathcal L}_X$ is a skew-adjoint operator.
For $t > 0$ and $\epsilon \in \C$, put
\begin{equation} \label{5.26}
D_{\alpha,t,\epsilon} = D_\alpha + \epsilon  \frac{c(X)}{4t}.
\end{equation}
As $c(X)$ is skew-adjoint, if $\epsilon$ is imaginary then
$D_{\alpha,t,\epsilon}$ is self-adjoint.
Put
\begin{align} \label{5.27}
F_{t,\epsilon}(X) = & \sum_\alpha
\Tr_s
\left( e^{-(t D_{\alpha,t,\epsilon}^2 + {\mathcal L}_X)}
\widetilde{\eta}_\alpha \right) + \\
& \frac12 \sum_\alpha \int_0^t \Tr_s
\left( e^{-(sD_{\alpha,t,\epsilon}^2 + {\mathcal L}_X)} D_{\alpha,t,\epsilon}
[D_{\alpha,t,\epsilon}, \widetilde{\eta}_\alpha] \right) \: ds. \notag
\end{align}
From Propositions \ref{proposition5} and \ref{proposition6},
\begin{equation} \label{5.28}
\Ind(D_{\inv}) = AV_X F_{t,0}(X).
\end{equation}

\begin{proposition} \label{proposition7}
$F_{t,0}(X)$ is independent of $t$.
\end{proposition}
\begin{proof}
We have
\begin{align} \label{5.29}
F_{t,0}(X) = & \sum_\alpha
\Tr_s
\left( e^{-(t D_{\alpha}^2 + {\mathcal L}_X)}
\widetilde{\eta}_\alpha \right) + \\
& \frac12 \sum_\alpha \int_0^t \Tr_s
\left( e^{-(sD_{\alpha}^2 + {\mathcal L}_X)} D_{\alpha}
[D_{\alpha}, \widetilde{\eta}_\alpha] \right) \: ds. \notag
\end{align}
Then
\begin{align} \label{5.30}
\frac{d}{dt} F_{t,0}(X) = & \sum_\alpha
- \Tr_s
\left( D_{\alpha}^2 \: e^{-(t D_{\alpha}^2 + {\mathcal L}_X)}
\widetilde{\eta}_\alpha \right) + \\
& \frac12 \sum_\alpha \Tr_s
\left( e^{-(tD_{\alpha}^2 + {\mathcal L}_X)} D_{\alpha}
[D_{\alpha}, \widetilde{\eta}_\alpha] \right) \: = \: 0. \notag
\end{align}
The proposition follows.
\end{proof}

\begin{proposition} \label{proposition8}
$F_{t,\epsilon}(X)$ is independent of $\epsilon$.
\end{proposition}
\begin{proof}
Let $[\cdot, \cdot]_+$ denote the anticommutator of two operators.
We have an identity of operators on $L^2(Y_\alpha, {\mathcal E}_{Y_\alpha})$ :
\begin{align} \label{5.31}
& e^{-(t D_{\alpha,t,\epsilon}^2 + {\mathcal L}_X)} \widetilde{\eta}_\alpha +
\frac12  \int_0^t
e^{-(sD_{\alpha,t,\epsilon}^2 + {\mathcal L}_X)} D_{\alpha,t,\epsilon}
[D_{\alpha,t,\epsilon}, \widetilde{\eta}_\alpha] \: ds \: = \\
& e^{-(t D_{\alpha,t,\epsilon}^2 + {\mathcal L}_X)} \widetilde{\eta}_\alpha +
\int_0^t
e^{-(sD_{\alpha,t,\epsilon}^2 + {\mathcal L}_X)} D_{\alpha,t,\epsilon}^2
\: \widetilde{\eta}_\alpha \: ds \:  - \notag \\
& \frac12 \int_0^t
[D_{\alpha,t,\epsilon}, e^{-(sD_{\alpha,t,\epsilon}^2 + {\mathcal L}_X)} D_{\alpha,t,\epsilon} \widetilde{\eta}_\alpha]_+ \: ds \: = \notag \\
& e^{-(t D_{\alpha,t,\epsilon}^2 + {\mathcal L}_X)} \widetilde{\eta}_\alpha -
\int_0^t \frac{d}{ds}
e^{-(sD_{\alpha,t,\epsilon}^2 + {\mathcal L}_X)}
\: \widetilde{\eta}_\alpha \: ds \:  - \notag \\
& \frac12 \int_0^t
[D_{\alpha,t,\epsilon}, e^{-(sD_{\alpha,t,\epsilon}^2 + {\mathcal L}_X)} D_{\alpha,t,\epsilon} \widetilde{\eta}_\alpha]_+ \: ds \: = \notag \\
& e^{-{\mathcal L}_X} \widetilde{\eta}_\alpha - \frac12 \int_0^t
[D_{\alpha,t,\epsilon}, e^{-(sD_{\alpha,t,\epsilon}^2 + {\mathcal L}_X)} D_{\alpha,t,\epsilon} \widetilde{\eta}_\alpha]_+ \: ds. \notag
\end{align}
Then
\begin{equation} \label{5.32}
F_{t,\epsilon}(X) \: = \:
\sum_\alpha \Tr_s \left(
e^{-{\mathcal L}_X} \widetilde{\eta}_\alpha - \frac12 \int_0^t
[D_{\alpha,t,\epsilon}, e^{-(sD_{\alpha,t,\epsilon}^2 + {\mathcal L}_X)} D_{\alpha,t,\epsilon} \widetilde{\eta}_\alpha]_+ \: ds
\right).
\end{equation}
In particular,
\begin{align} \label{5.33}
\frac{d}{d\epsilon} F_{t,\epsilon}(X) \: = \: & - \frac12
\sum_\alpha \Tr_s \left[ \frac{d}{d\epsilon} D_{\alpha,t,\epsilon}, \int_0^t
e^{-(sD_{\alpha,t,\epsilon}^2 + {\mathcal L}_X)} D_{\alpha,t,\epsilon} \widetilde{\eta}_\alpha \: ds
\right]_+ \: - \\
& \frac12
\sum_\alpha \Tr_s \left[ D_{\alpha,t,\epsilon}, \frac{d}{d\epsilon} \left( \int_0^t
e^{-(sD_{\alpha,t,\epsilon}^2 + {\mathcal L}_X)} D_{\alpha,t,\epsilon} \widetilde{\eta}_\alpha \: ds
\right)
\right]_+ \notag \\
\: = \: & 0. \notag
\end{align}
The proposition follows.
\end{proof}

\begin{corollary} \label{corollary10}
$F_{t,\epsilon}(X)$ is independent of $t$ and $\epsilon$.
\end{corollary}
\begin{proof}
This follows from Propositions \ref{proposition7} and \ref{proposition8}.
\end{proof}

\begin{proposition} \label{proposition9}
$F_{t,2}(X)$ has a holomorphic extension to $X \in \C^k$.
\end{proposition}
\begin{proof}
One finds
\begin{equation} \label{explicit}
t \: D_{\alpha,t,2}^2 \: + \: {\mathcal L}_X \: = \:
t D_\alpha^2 + \mu(X) + \frac12 c \left( dX^* \right) - \frac{X^2}{4t}.
\end{equation}
Writing
\begin{align} \label{Duhamel}
F_{t,2}(X) = & \sum_\alpha
\Tr_s
\left( e^{-(t D_{\alpha,t,2}^2 + {\mathcal L}_X)}
\widetilde{\eta}_\alpha \right) + \\
& \frac12 \sum_\alpha \int_0^t \Tr_s
\left( e^{-(sD_{\alpha,t,2}^2 + {\mathcal L}_X)} D_{\alpha,t,2}
[D_{\alpha,t,2}, \widetilde{\eta}_\alpha] \right) \: ds, \notag
\end{align}
and using (\ref{explicit}), we expand the right-hand side of
(\ref{Duhamel}) by means of a Duhamel expansion.
The estimates of \cite[Lemma 2.1]{Getzler-Szenes (1989)} show
that the ensuing series defines a holomorphic function of $X \in \C^k$.
\end{proof}

As a consequence of Corollary \ref{corollary10} and Proposition \ref{proposition9},
for any $t > 0$ and $\epsilon \in \C$,
$F_{t,\epsilon}(X)$ has a holomorphic extension to $X \in \C^k$.

\subsection{Short-time delocalized limit} \label{subsection5.3}

Let $\widehat{A}(X, Y_\alpha) \: \ch(X, {\mathcal E}_\alpha/S)
\in \Omega^*(Y_\alpha)$
be the equivariant characteristic form defined in
\cite[Chapter 8.1]{Berline-Getzler-Vergne (1992)}.
Notationally,
\begin{equation}
\widehat{A}(X, Y_\alpha) = \sqrt{ \det \left(
\frac{
R_{\mathfrak g}(X)/2
}{
\sinh \left( R_{\mathfrak g}(X)/2 \right)
}
\right)
},
\end{equation}
with $R_{\mathfrak g}(X) = R + \mu(X)$,
and
\begin{equation}
\ch(X, {\mathcal E}_\alpha/S) = \tr_{{\mathcal E}_\alpha/S}
\left( e^{- \: F_{\mathfrak g}^{{\mathcal E}_\alpha/S}(X)} \right).
\end{equation}
Note that $\widehat{A}(X, Y_\alpha) \: \ch(X, {\mathcal E}_\alpha/S)$
has an analytic extension to $\C^k$ which is regular in a neighborhood
of $0$, and on the complement of $\R^k$.

\begin{proposition} \label{proposition10}
If $X \in \R^k$ then
\begin{equation} \label{5.34}
\lim_{t \rightarrow 0} F_{t,1}(iX) = \sum_\alpha \int_{Y_\alpha}
\widehat{A}(iX, Y_\alpha) \: \ch(iX, {\mathcal E}_\alpha/S) \:
\widetilde{\eta}_\alpha.
\end{equation}
\end{proposition}
\begin{proof}
We can write
\begin{align} \label{5.35}
F_{t,1}(iX) = & \sum_\alpha
\Tr_s
\left( e^{-(t D_{\alpha,t,i}^2 + i{\mathcal L}_X)} \widetilde{\eta}_\alpha \right) + \\
& \frac12 \sum_\alpha \int_0^t \Tr_s
\left( e^{-(sD_{\alpha,t,i}^2 + i{\mathcal L}_X)} D_{\alpha,t,i}
[D_{\alpha,t,i}, \widetilde{\eta}_\alpha] \right) \: ds. \notag
\end{align}
Note that
$t D_{\alpha,t,i}^2 + i{\mathcal L}_X$ is a self-adjoint operator.
Now
\begin{equation} \label{5.36}
\Tr_s
\left( e^{-(t D_{\alpha,t,i}^2 + i{\mathcal L}_X)} \widetilde{\eta}_\alpha \right) =
\int_{Y_\alpha} \trs
\left( e^{-(t D_{\alpha,t,i}^2 + i{\mathcal L}_X)} \right)(p,p) \:
\widetilde{\eta}_\alpha(p) \: d\mu_{Y_\alpha}(p).
\end{equation}
From \cite[Section 2]{Bismut (1985)},
\begin{equation} \label{5.37}
\lim_{t \rightarrow 0} \left( e^{-(t D_{\alpha,t,i}^2 + i{\mathcal L}_X)} \right)(p,p)
\: = \: \left( \widehat{A}(iX, Y_\alpha) \:
\ch(iX, {\mathcal E}_\alpha/S) \right)(p).
\end{equation}
Thus
\begin{equation} \label{5.38}
\lim_{t \rightarrow 0} \sum_\alpha
\Tr_s
\left( e^{-(t D_{\alpha,t,i}^2 + i{\mathcal L}_X)} \widetilde{\eta}_\alpha \right) \: = \:
\sum_\alpha \int_{Y_\alpha} \widehat{A}(iX, Y_\alpha) \:
\ch(iX, {\mathcal E}_\alpha/S) \: \widetilde{\eta}_\alpha.
\end{equation}

Next, we want to show that
\begin{equation} \label{5.39}
\lim_{t \rightarrow 0} \frac12 \: \sum_\alpha \int_0^t \Tr_s
\left( e^{-(sD_{\alpha,t,i}^2 + i{\mathcal L}_X)} D_{\alpha,t,i}
[D_{\alpha,t,i}, \widetilde{\eta}_\alpha] \right) \: ds \: = \: 0.
\end{equation}
For this, we have to show certain cancellations between the
terms for various $\alpha$.

Define a measure $\nu_t$ on $W$ by
\begin{equation} \label{5.40}
\nu_t = \sum_\alpha (\pi_\alpha)_* \left(
\frac12 \int_0^t \trs
\left( e^{-(sD_{\alpha,t,i}^2 + i{\mathcal L}_X)} D_{\alpha,t,i}
[D_{\alpha,t,i}, \widetilde{\eta}_\alpha] \right)(p,p) \: d\mu_{Y_\alpha}(p) \: ds
\right).
\end{equation}
We want to show that the integral of $\nu_t$ vanishes as $t \rightarrow 0$.

Given $w \in W$, choose a point $\widetilde{p} \in {\mathcal T}$ that
projects to $w$. Let $\widetilde{K}$ be the isotropy group of
$\overline{{\mathcal G}_{\mathcal T}}$ at $\widetilde{p}$.
For each $\alpha$ with $w \in U_\alpha$, choose
$p_\alpha \in Y_\alpha$ projecting to $w$. By the slice theorem,
there is a neighborhood of $w$ in $W$ homeomorphic to
$B(\widetilde{V})/\widetilde{K}$, where $\widetilde{V}$ is a
representation space of $\widetilde{K}$ and
$B(\widetilde{V})$ is a ball in $\widetilde{V}$.
There is a neighborhood of $\widetilde{p}$ which, for each $\alpha$,
is isometric to a neighborhood of $p_\alpha$. We will use this to
identify each $p_\alpha$ with $\widetilde{p}$.

Using Example \ref{example8},
\begin{align} \label{5.41}
\nu_t(w) \: = & \:
\left( \sum_\alpha
\frac12 \int_0^t \trs
\left( e^{-(sD_{\alpha,t,i}^2 + i{\mathcal L}_X)} D_{\alpha,t,i}
[D_{\alpha,t,i}, \widetilde{\eta}_\alpha] \right)(\widetilde{p},\widetilde{p}) \:
ds \right) \: h_W(w) \: d\mu_W(w) \\
= & \: \sum_{\alpha,\beta}
\frac12 \int_0^t \left[ \trs
\left( e^{-(sD_{\alpha,t,i}^2 + i{\mathcal L}_X)} D_{\alpha,t,i}
[D_{\alpha,t,i}, \widetilde{\eta}_\alpha] \right)(\widetilde{p},\widetilde{p}) \: -
\right. \notag \\
& \left. \trs
\left( e^{-(sD_{\beta,t,i}^2 + i{\mathcal L}_X)} D_{\beta,t,i}
[D_{\alpha,t,i}, \widetilde{\eta}_\alpha] \right)(\widetilde{p},\widetilde{p}) \right] \: \eta_\beta(w) \:
ds \: h_W(w) \: d\mu_W(w). \notag
\end{align}
As $D_{\alpha,t,i}$ coincides with $D_{\beta,t,i}$ in a neighborhood of
$\widetilde{p}$, under our identifications, it follows
from finite propagation speed estimates
\cite{Cheeger-Gromov-Taylor (1982)}
that $\frac{\nu_t(w)}{h_W(w) \: d\mu_W(w)}$ decays as
$t \rightarrow 0$ faster than any power of $t$.
These estimates can clearly be made uniform in $w$.
The proposition follows.
\end{proof}

We now prove a delocalized index theorem.

\begin{corollary} \label{corollary11}
\begin{equation} \label{5.42}
\Ind(D_{\inv}) = AV_X
\sum_\alpha \int_{Y_\alpha}
\widehat{A}(X, Y_\alpha) \: \ch(X, {\mathcal E}_\alpha/S) \:
\widetilde{\eta}_\alpha.
\end{equation}
\end{corollary}
\begin{proof}
As in (\ref{5.28}),
$\Ind \left( D_{\inv} \right) = AV_X F_{t,0}(X)$.
By Corollary \ref{corollary10} and Proposition \ref{proposition9},
$F_{t,0}(X)$ has an holomorphic extension to $\C^k$.
By Corollary \ref{corollary10} and Proposition \ref{proposition10},
if $X \in i \R^k$ then
\begin{equation} \label{analytic}
F_{t,0}(X) = \sum_\alpha \int_{Y_\alpha}
\widehat{A}(X, Y_\alpha) \: \ch(X, {\mathcal E}_\alpha/S) \:
\widetilde{\eta}_\alpha.
\end{equation}
By analytic continuation, (\ref{analytic}) holds for $X \in \C^k$.
The corollary follows.
\end{proof}

\begin{remark}
Although $\int_{Y_\alpha}
\widehat{A}(X, Y_\alpha) \: \ch(X, {\mathcal E}_\alpha/S) \:
\widetilde{\eta}_\alpha$ may have singularities in $X$ for
individual $\alpha$, the
proof of Corollary \ref{corollary11} shows that the sum over $\alpha$ is
holomorphic in $X$.
\end{remark}

\section{Local index formula and applications} \label{section6}

In this section we prove the main theorem of the paper.
In Subsection \ref{subsection5.4} we
localize the index theorem of the previous section
to the fixed-point sets.  In Subsection \ref{local} we prove the
index theorem stated in the introduction of the paper.
In Subsection \ref{computing}
we describe how to compute the terms appearing in the local index formula.
We carry out the computation when $D$ is the pure Dirac operator,
the signature operator and the Euler operator.

\subsection{Localization to the fixed-point set} \label{subsection5.4}

Let ${\mathcal T}^{T^k}$ be the subset of ${\mathcal T}$ consisting
of points with isotropy group isomorphic to $T^k$.
Let $\{Z_i^{T^k}\}$ be the connected components of
$\sigma\left( {\mathcal T}^{T^k} \right) \subset W$. From our assumptions,
each $Z_i^{T^k}$ is a smooth manifold.  Furthermore, the
Clifford module ${\mathcal E}$ on ${\mathcal T}$ descends to a
$T^k$-equivariant Clifford module ${\mathcal E}_i$ on $Z_i^{T^k}$.
There is a natural vector bundle $N_i$ on $Z_i^{T^k}$ so that for
$w \in Z_i^{T^k}$, if we choose $p \in \sigma^{-1}(w) \in {\mathcal T}$
then the fiber $(N_i)_w$ is isomorphic to the normal bundle of
${\mathcal T}^{T^k}$ in ${\mathcal T}$ at $p$.
The bundle $N_i$ inherits an orthogonal connection.
Let $R_{N_i}$ denote its curvature $2$-form.

For simplicity, we assume that ${\mathcal T}$ has a
${\mathcal G}_{\mathcal T}$-invariant spin structure,
with spinor bundle $S^{\mathcal T}$,
and that
${\mathcal E} \: = \: S^{\mathcal T} \otimes {\mathcal W}$ for
some $\Z_2$-graded ${\mathcal G}_{\mathcal T}$-equivariant
vector bundle ${\mathcal W}$. Suppose further that each
$Z_i^{T^k}$ is spin. We can define the normal spinor bundle
$S_N$ on $Z_i^{T^k}$.

Let $e^{-X} \in T^k$ denote the exponential of $-X \in {\mathfrak g}$.

\begin{proposition} \label{proposition11}
\begin{equation} \label{5.43}
AV_X \sum_\alpha \int_{Y_\alpha}
\widehat{A}(X, Y_\alpha) \: \ch(X, {\mathcal E}_\alpha/S) \:
\widetilde{\eta}_\alpha \: = \:
AV_X \sum_i \int_{Z^{T^k}_i}
\widehat{A}(TZ^{T^k}_i) \: \frac{\ch_{\mathcal W}(e^{-X})}{\ch_{S_N}(e^{-X})}.
\end{equation}
\end{proposition}
\begin{proof}
Let $Z(X)$ denote the zero-set of $X$ on $\coprod_\alpha Y_\alpha$.
As in \cite[Chapter 7.2]{Berline-Getzler-Vergne (1992)},
away from $Z(X)$ we can write
\begin{align} \label{5.44}
\widehat{A}(X, Y_\alpha) \: \ch(X, {\mathcal E}_\alpha/S) \:
\widetilde{\eta}_\alpha \: = \: &
d_X \left( \frac{X^* \wedge
\widehat{A}(X, Y_\alpha) \: \ch(X, {\mathcal E}_\alpha/S) \:
\widetilde{\eta}_\alpha}
{d_X X^*} \right) \: + \\
& \frac{X^* \wedge
\widehat{A}(X, Y_\alpha) \: \ch(X, {\mathcal E}_\alpha/S) \:
}
{d_X X^*} \wedge d_X \widetilde{\eta}_\alpha. \notag
\end{align}
This formula extends analytically to $X$ lying in a
suitable neighborhood of the origin in $\C^k$.
Then because $\sum_\alpha \widetilde{\eta}_\alpha \: = \: 1$,
the localization argument in the proof of
\cite[Theorem 7.13]{Berline-Getzler-Vergne (1992)} applies to give
\begin{equation} \label{5.45}
\sum_\alpha \int_{Y_\alpha}
\widehat{A}(X, Y_\alpha) \: \ch(X, {\mathcal E}_\alpha/S) \:
\widetilde{\eta}_\alpha \: = \:
\sum_\alpha
\int_{Z(X)}
\widehat{A}(TZ(X)) \: \frac{\ch_{\mathcal W}(e^{-X})}{\ch_{S_N}(e^{-X})}
\: \widetilde{\eta}_\alpha.
\end{equation}
Because the left-hand side of (\ref{5.45}) has a holomorphic extension
to $\C^k$, the same is true for the right-hand side.  So the
formula makes sense for $X \in \R^K$.

When we average over $X \in \R^k$, the integral over a component of
$Z(X)$ will not contribute unless the component lies in
$\bigcap_{X^\prime \in \R^k} Z(X^\prime)$.
Hence
\begin{equation} \label{5.46}
AV_X \sum_\alpha \int_{Y_\alpha}
\widehat{A}(X, Y_\alpha) \: \ch(X, {\mathcal E}_\alpha/S) \:
\widetilde{\eta}_\alpha \: = \:
AV_X \sum_\alpha
\int_{\bigcap_{X^\prime} Z(X^\prime)}
\widehat{A}(TZ(X)) \: \frac{\ch_{\mathcal W}(e^{-X})}{\ch_{S_N}(e^{-X})}
\: \widetilde{\eta}_\alpha.
\end{equation}
We can identify the image of $\bigcap_{X^\prime} Z(X^\prime)$, under
the projection map $\coprod_\alpha \left( B(V_\alpha) \times_{K_\alpha}
G_\alpha \right) \rightarrow W$, with
$\bigcup_i Z_i^{T^k}$. After making this identification, the
proposition follows.
\end{proof}

\begin{remark}
It follows from the proof of Proposition \ref{proposition11} that
$\sum_i \int_{Z^{T^k}_i}
\widehat{A}(TZ^{T^k}_i) \: \frac{\ch_{\mathcal W}(e^{-X})}{\ch_{S_N}(e^{-X})}$
is holomorphic in $X \in \C^k$.
Each term
$\int_{Z^{T^k}_i}
\widehat{A}(TZ^{T^k}_i) \: \frac{\ch_{\mathcal W}(e^{-X})}{\ch_{S_N}(e^{-X})}$
is meromorphic in $X \in \C^k$.
\end{remark}

\begin{corollary} \label{corollary12}
For any $Q \in \C^k$,
\begin{equation} \label{5.47}
\Index(D_{\inv}) \: = \: AV_X \sum_i \int_{Z^{T^k}_i}
\widehat{A}(TZ^{T^k}_i) \:
\frac{\ch_{\mathcal W}(e^{-X+Q})}{\ch_{S_N}(e^{-X+Q})}.
\end{equation}
\end{corollary}
\begin{proof}
The integral $\int_{Z^{T^k}_i}
\widehat{A}(TZ^{T^k}_i) \:
\frac{\ch_{\mathcal W}(e^{-X})}{\ch_{S_N}(e^{-X})}$ is a meromorphic
function in $X \in \C^k$ which is invariant with respect to a lattice $L_i
\subset \R^k$.
As the sum over $i$ is holomorphic, it follows that we can write
$\sum_i \int_{Z^{T^k}_i}
\widehat{A}(TZ^{T^k}_i) \:
\frac{\ch_{\mathcal W}(e^{-X})}{\ch_{S_N}(e^{-X})}$ as a finite sum
$\sum_j H_j(X)$, where each $H_j$ is a holomorphic function of
$X \in \C^k$ that is invariant 
with respect to a lattice $L_j \subset \R^k$. 
Now $AV_X H_j(X)$ can be computed 
by means of a product of contour integrals in $\C^k$. Computing instead
$AV_X H_j(X-Q)$ amounts to deforming the contours. Hence
$AV_X H_j(X-Q) = AV_X H_j(X)$, from which the corollary follows.
\end{proof}

\subsection{Local index formula} \label{local}

We will need the explicit
formula for $\frac{1}{\ch_{S_N}(e^{-X+Q})}$.
Given $z \neq 1$ and a complex $r$-dimensional vector bundle $L$,
put
\begin{equation} \label{4.21}
{\mathcal F}_{\Dirac}(L,z) \: = \: \prod_{j=1}^r \left(
z^{-\frac12} \: e^{\frac{x_j}{2}} \: - \:
z^{\frac12} \: e^{-\frac{x_j}{2}} \right)^{-1},
\end{equation}
where the $x_j$'s are the formal roots of the total Chern class of $L$.
As usual, the expression (\ref{4.21}) is meant to be expanded in the
$x_j$'s, which have formal degree two.

Let $Z_i^{T^k}$ and $N_i$ be as before.
Suppose that with respect to the $\R^k$-action,
$N_i$ is isomorphic to the underlying real bundle of a
direct sum of complex line bundles
$\bigoplus_q N_{q,i}$,
where $e^{-X}$ acts on $N_{q,i}$
by $e^{- \sqrt{-1} {\bf n}_{q,i} \cdot X}$ for some
${\bf n}_{q,i} \in \R^k$.
Then
\begin{equation} \label{4.22}
\frac{1}{\ch_{S_N}(e^{-X+Q})} \: = \: \pm \:
\prod_q {\mathcal F}_{\Dirac} \left( N_{q,i},
e^{- \sqrt{-1} {\bf n}_{q,i} \cdot (X-Q)} \right).
\end{equation}
See \cite{Atiyah-Hirzebruch (1970)} for a discussion of the sign issue.

The individual term
$\int_{Z_i^{T^k}}
\widehat{A}(TZ_i^{T^k}) \:
\frac{\ch_{\mathcal W}(e^{-X+Q})}{\ch_{S_N}(e^{-X+Q})}$ is smooth
in $X$ provided that $\Image(Q) \notin \bigcup_q {\bf n}_{q,i}^\perp$.

Let $W_{\max}$ denote the image of $\bigcup_i Z_i^{T^k}$
under
the projection map $\coprod_\alpha \left( B(V_\alpha) \times_{K_\alpha}
G_\alpha \right) \rightarrow W$.  It is a smooth manifold and is
the deepest stratum in $W$,
with respect to the partial ordering described in
\cite[Section 3.3]{Haefliger (1985)}.
Note that $W_{\max}$ could be the empty set.

Suppose that ${\mathcal E} = S^{\mathcal T} \otimes W$ and that
$W_{\max}$ is spin.

\begin{definition}
If $\Image(Q) \notin \bigcup_i \bigcup_q {\bf n}_{q,i}^\perp$,
define ${\mathcal N}_{{\mathcal E}, Q} \in \Omega^* W_{\max}$ by
\begin{equation} \label{Nformula}
{\mathcal N}_{{\mathcal E}, Q} \: = \: AV_X
\frac{
\ch_{\mathcal W}(e^{-X+Q})
}{
\ch_{\mathcal S_N}(e^{-X+Q})
}.
\end{equation}
\end{definition}

\begin{theorem} \label{theorem2}
\begin{equation}
\Index(D_{\inv}) \: = \: \int_{W_{\max}}
\widehat{A}(TW_{\max}) \: {\mathcal N}_{{\mathcal E}, Q}.
\end{equation}
\end{theorem}
\begin{proof}
This follows from Corollary \ref{corollary12}.
\end{proof}

We now remove the assumptions that
${\mathcal E} = S^{\mathcal T} \otimes W$ and
$W_{\max}$ is spin.
We use the notation of \cite[Chapter 6.4]{Berline-Getzler-Vergne (1992)}.

\begin{definition} \label{generaldef}
If $\Image(Q) \notin \bigcup_i \bigcup_q {\bf n}_{q,i}^\perp$,
define ${\mathcal N}_{{\mathcal E}, Q} \in \Omega^* W_{\max}$ by
\begin{equation} \label{5.49}
{\mathcal N}_{{\mathcal E}, Q} \: = \: AV_X
\frac{
\ch_{{\mathcal E}/S_N}(e^{-X+Q})
}{
\sqrt{\det \left( 1 - e^{-X+Q} \cdot e^{-R_N} \right)}
}.
\end{equation}
\end{definition}

\begin{theorem} \label{theorem3}
\begin{equation} \label{5.50}
\Index(D_{\inv}) \: = \: \int_{W_{\max}}
\widehat{A}(TW_{\max}) \: {\mathcal N}_{{\mathcal E}, Q}.
\end{equation}
\end{theorem}
\begin{proof}
If ${\mathcal E} \: = \: S^{\mathcal T} \otimes {\mathcal W}$
and $W_{\max}$ is spin
then from
\cite[Chapter 6.4]{Berline-Getzler-Vergne (1992)},
\begin{equation}
\frac{
\ch_{{\mathcal E}/S_N}(e^{-X+Q})
}{
\sqrt{\det \left( 1 - e^{-X+Q} \cdot e^{-R_N} \right)}
} = \frac{\ch_{\mathcal W}(e^{-X+Q})}{\ch_{S_N}(e^{-X+Q})}.
\end{equation}
Hence in this case, the theorem reduces to Theorem \ref{theorem2}.
The general case can be proved by means similar to the proof of
Theorem \ref{theorem2}, carrying along the more general assumptions
throughout.
\end{proof}

Theorem \ref{theorem3} implies Theorem \ref{theorem1}, because of our assumption in
Theorem \ref{theorem1} that the Molino sheaf acts on the Clifford module ${\mathcal E}$
(which lives on $M$).
More precisely, we are assuming that the restriction ${\mathcal E}_{\mathcal T}$
of  ${\mathcal E}$ to ${\mathcal T}$ carries a representation of the Lie
algebroid $\overline{{\mathfrak g}_{\mathcal T}}$ in the sense of
\cite[Section 1.4]{Crainic (2003)}. Then ${\mathcal E}_{\mathcal T}$ is a
$\overline{{\mathcal G}_{\mathcal T}}$-equivariant vector bundle on ${\mathcal T}$ and
Theorem \ref{theorem3} applies.

\begin{remark} \label{remark4}
If $M$ is a simply-connected manifold with a Riemannian foliation
then its space $W$ of leaf closures is the quotient of an orbifold $Y$
by a $T^N$-action \cite{Haefliger-Salem (1990)}. One might
hope to reduce the computation of the index of a basic Dirac-type operator
on $M$ to the computation of the $T^N$-invariant index of a
Dirac-type operator on $Y$. Unfortunately, the \'etale groupoid
$Y \rtimes T^N_\delta$ is generally not
weak equivalent to $\overline{{\mathcal G}_{\mathcal T}}$ with its
\'etale topology. In general
$\dim(Y) > \dim({\mathcal T})$, so there is no associated Dirac-type
operator on $Y$.
\end{remark}

\subsection{Computing the index} \label{computing}

For simplicity, we assume again that
${\mathcal E} = S^{\mathcal T} \otimes {\mathcal W}$ (which is
always the case locally) and that $W_{\max}$ is spin,
so that we have the simpler
formula (\ref{Nformula}) for ${\mathcal N}_{{\mathcal E}, Q}$.

The action of $\{e^{-X}\}$ on $S_N$ and ${\mathcal W}$, over
a connected component $Z_i^{T^k}$ of $W_{\max}$,
factors through an action of $T^k$.
Because of this $T^k$-action,
we can compute $AV_X$ by performing the contour integral over
$(S^1)^k \subset \C^k$ of a certain rational function times
$\frac{dz_1}{2\pi \sqrt{-1} z_1} \ldots \frac{dz_k}{2\pi \sqrt{-1} z_k}$.
The result depends
{\it a priori} on $Q$
(recall that $\Image(Q) \notin \bigcup_i \bigcup_q {\bf n}_{q,i}^\perp$)
although of course the final answer for the index is independent of $Q$.

Changing $Q$ amounts to deforming the contour of integration in
$\C^k$.
Hence the local formula for
$\Index(D_{\inv})$ depends on $Q$ through the chamber of
$\bigcap_i \bigcap_q (\R^k - {\bf n}_{q,i}^\perp)$ to which
$\Image(Q)$ belongs.
Passing from one chamber to another one, the local formula could
{\it a priori} change.  This is not surprising, in view of the cancellations
of singularities that occur; one could add various local contributions
to the index formula, which will cancel out in the end.

We now apply Theorem \ref{theorem1} to some geometric Dirac-type operators,
in which case the action of the Molino sheaf on ${\mathcal E}$ is automatic.

\subsubsection{Pure Dirac operator} \label{example10}

\begin{proposition} \label{corollary13}
Suppose that $D$ is the pure Dirac operator.  Then
$\Index(D_{\inv})$ vanishes if $k > 0$, while
\begin{equation} \label{5.51}
\Index(D_{\inv}) \: = \: \widehat{A}(W)
\end{equation}
if $k =0$.
\end{proposition}
\begin{proof}
From Corollary \ref{corollary12},
\begin{equation}
\Index(D_{\inv}) \: = \: AV_X \sum_i \int_{Z^{T^k}_i}
\widehat{A}(TZ^{T^k}_i) \:
\frac{
1
}{
\ch_{S_N}(e^{-X+Q})
}.
\end{equation}
Take $Q$ so that $\Image(Q) \in
\bigcap_i \bigcap_q (\R^k - {\bf n}_{q,i}^\perp)$.
Consider the effect of multiplying $Q$ by $\lambda > 0$.
Each factor in (\ref{4.21}) has a term of either $z^{- \: \frac12}$
or $z^\frac12$, appearing in the denominator. It follows that
as $\lambda \rightarrow \infty$, the right-hand side of (\ref{4.22})
decreases exponentially fast in $\lambda$. Thus if $k > 0$ then
$\Index(D_{\inv}) = 0$. If $k = 0$ then the foliated manifold $M$
is the total space of a fiber bundle over $W = W_{\max}$ and $D_{\inv}$ is
conjugate to the pure Dirac operator on $W$, so
$\Index(D_{\inv}) = \widehat{A}(W)$.
\end{proof}

\subsubsection{Signature operator} \label{example11}

\begin{proposition} \label{corollary14}
Suppose that ${\mathcal F}$ is transversely oriented and
$\dim({\mathcal T})$ is divisible by four.
Recall the notion of the basic signature
$\sigma(M, {\mathcal F}; {\mathcal D}_M^{\frac12})$
from Subsection \ref{subsection2.7}.
We have
\begin{equation} \label{5.52}
\sigma(M, {\mathcal F}; {\mathcal D}_M^{\frac12})
\: = \:  \sigma(W_{\max}).
\end{equation}
\end{proposition}
\begin{proof}
From Corollary \ref{corollary8},
$\sigma(M, {\mathcal F}; {\mathcal D}_M^{\frac12})$ equals the index
of $D_{\inv}$ when $D$ is the operator
$d + d^*$ and the $\Z_2$-grading comes from the Hodge duality operator.
A component $Z^{T^k}_i$ of $W_{\max}$ acquires a natural orientation.
Given $z \neq 1$ and a complex $r$-dimensional vector bundle $L$,
put
\begin{equation} \label{4.25}
{\mathcal F}_{\sign}(L,z) \: = \: \prod_{j=1}^r \frac{
z^{-\frac12} \: e^{\frac{x_j}{2}} \: + \:
z^{\frac12} \: e^{-\frac{x_j}{2}} }{
z^{-\frac12} \: e^{\frac{x_j}{2}} \: - \:
z^{\frac12} \: e^{-\frac{x_j}{2}} }.
\end{equation}
Then
\begin{equation} \label{Leqn}
\Index(D_{\inv}) \: = \: AV_X \sum_i \int_{Z^{T^k}_i}
L(TZ^{T^k}_i) \:
\Phi(e^{-X+Q}),
\end{equation}
where
\begin{equation}
\Phi(e^{-X+Q}) \: = \: \pm \:
\prod_q {\mathcal F}_{\sign} \left( N_{q,i},
e^{- \sqrt{-1} {\bf n}_{q,i} \cdot (X-Q)} \right).
\end{equation}

Take $Q$ so that $\Image(Q) \in \bigcap_i \bigcap_q
(\R^k - {\bf n}_{q,i}^\perp)$.
Consider the effect of multiplying $Q$ by $\lambda > 0$.
From the structure of (\ref{4.25}), and taking the signs into account,
the limit as $\lambda \rightarrow \infty$ of
$\Phi(e^{-X+ \lambda Q})$ is $1$.
Thus $\Index(D_{\inv}) = \sum_i \int_{Z^{T^k}_i} L(TZ^{T^k}_i)$,
which equals the signature of $W_{\max}$.
\end{proof}

\subsubsection{Euler operator} \label{example12}

\begin{proposition} \label{corollary15}
Suppose that $\dim({\mathcal T})$ is even.
Recall the notion of the basic Euler characteristic
$\chi(M, {\mathcal F}; {\mathcal D}_M^{\frac12})$
from Subsection \ref{subsection2.7}.
We have
\begin{equation} \label{5.53}
\chi(M, {\mathcal F}; {\mathcal D}_M^{\frac12}) \: =
\: \chi(W_{\max}).
\end{equation}
\end{proposition}
\begin{proof}
From Corollary \ref{corollary8},
$\chi(M, {\mathcal F}; {\mathcal D}_M^{\frac12})$ equals the index
of $D_{\inv}$ when $D$ is the operator
$d + d^*$ and the $\Z_2$-grading comes from the form degree.
Then
\begin{equation}
\Index(D_{\inv}) \: = \: AV_X \sum_i \int_{Z^{T^k}_i}
e(TZ^{T^k}_i),
\end{equation}
where
$e$ denotes the Euler form.
Thus $\Index(D_{\inv}) = \sum_i \chi(Z^{T^k}_i)$,
which equals the Euler characteristic of $W_{\max}$.
\end{proof}

\bibliographystyle{acm}

\end{document}